\documentclass[11pt]{amsart}
\usepackage{amsfonts} 
\usepackage{amssymb}
\usepackage{amsthm}
\usepackage{amsmath} 
\usepackage{mathrsfs}
\usepackage{mathtools}
\usepackage{dsfont}
\usepackage{esint}
\usepackage{ulem}
\usepackage{marginnote}
\usepackage{array}
\usepackage{algorithm}
\usepackage{algpseudocode}
\usepackage{soul}
\usepackage{comment}

\usepackage{booktabs} 

\usepackage{MnSymbol}

\usepackage{longtable}

\usepackage[all]{xy}

\usepackage[english]{babel} 
\usepackage[left=3.5cm,right=3.5cm,top=3.5cm,bottom=3cm,headsep=0.7cm]{geometry}
\usepackage{xargs}

\usepackage{csquotes}
\usepackage{stmaryrd}

\usepackage{pgf,tikz}
\usetikzlibrary{positioning, calc, arrows.meta}
\usetikzlibrary{arrows}
\usepackage{subcaption}
\usepackage{graphicx}
\usepackage[parfill]{parskip}

\usepackage{enumerate}
\usepackage{enumitem}

\usepackage{hyperref}
\usepackage[nameinlink]{cleveref}
\usepackage{mathtools}
\hypersetup{
	colorlinks,
	linkcolor={blue!80!black},
	citecolor={blue!80!black}
}

\theoremstyle{definition}
\newtheorem{definition}{Definition}[section]
\newtheorem{openproblem}{Open Problem}[section]

\theoremstyle{plain}
\newtheorem{theorem}{Theorem}[section]
\newtheorem{proposition}{Proposition}[section]
\newtheorem{lemma}{Lemma}[section]
\newtheorem{corollary}{Corollary}[section]

\theoremstyle{remark}
\newtheorem{remark}{Remark}[section]

\newcommand{\SCI}{\operatorname{SCI}}
\newcommand{\SCIG}{\operatorname{SCI}_{\mathrm G}}

\newcommand{\fq}{\mathrm{fq}}
\newcommand{\raw}{\mathrm{raw}}
\newcommand{\geom}{\mathrm{geom}}
\newcommand{\rep}{\mathrm{rep}}
\newcommand{\cont}{\mathrm{cont}}
\newcommand{\Bor}{\mathrm{Bor}}
\newcommand{\TTE}{\mathrm{TTE}}
\newcommand{\diag}{\mathrm{diag}}
\newcommand{\Graph}{\mathrm{graph}}
\newcommand{\gen}{\mathrm{gen}}
\newcommand{\Sp}{\sigma}

\newcommand{\cK}{\mathcal K}
\newcommand{\cO}{\mathcal O}

\newcommand{\bN}{\mathbb N}
\newcommand{\bC}{\mathbb C}
\newcommand{\bR}{\mathbb R}
\newcommand{\bQ}{\mathbb Q}

\newcommand{\leGfq}{\le_{\mathrm G,\fq}}
\newcommand{\lemfq}[1]{\le_{#1,\fq}}
\newcommand{\Nlemfq}[1]{\not{\le}_{#1,\fq}}

\newcommand{\Psd}{\mathcal P^{\sigma}_{J,\diag}}
\newcommand{\Psg}{\mathcal P^{\sigma}_{J,\gen}}
\newcommand{\Psr}{\mathcal P^{\sigma}_{J,\Graph}}
\newcommand{\Ppd}{\mathcal P^{\sigma_\varepsilon}_{J,\varepsilon,\diag}}
\newcommand{\Ppg}{\mathcal P^{\sigma_\varepsilon}_{J,\varepsilon,\gen}}
\newcommand{\Ppr}{\mathcal P^{\sigma_\varepsilon}_{J,\varepsilon,\Graph}}

\newcommand{\BorD}{\mathrm{Bor}_{D}}
\newcommand{\BorcD}{\mathrm{Bor}_{E,\vartheta;D=\mathrm{cont}}}
\newcommand{\Borfull}{\mathrm{Bor}_{\mathrm{full}}}
\newcommand{\GeomTTE}{\geom^{\TTE}}
\newcommand{\GeomCont}{\geom^{\cont}}
\newcommand{\RepTTE}{\rep^{\TTE}}
\newcommand{\RepCont}{\rep^{\cont}}

\title[\bf Finite-Query Collapse and Modal Exact Bases in the SCI Hierarchy]{\bf Finite-Query Collapse and Modal Exact Bases in the SCI Hierarchy}

\begin{document}
\author[C.~Sorg]{Christopher Sorg$^1$}
\address[C.~Sorg]{
	\textup{Chair for Theoretical computer science, mathematics, and operations research}
	\newline \indent
	\textup{Department of Computer Science} \newline \indent
	\textup{University of the Bundeswehr Munich}
	\newline \indent
	\textup{85577 Neubiberg, Germany}}

\email{{\href{mailto:chr.sorg@unibw.de}{\textcolor{blue}{\texttt{chr.sorg@unibw.de}}}}
}

\footnotetext[1]{Inf1, University of the Bundeswehr Munich, Werner-Heisenberg-Weg 39, 85577 Neubiberg, Germany}

\begin{abstract}
	We study the exact-basis problem for Solvability Complexity Index (SCI) computational problem families through finite-query transports. A raw finite-query reduction permits arbitrary encodings and finite transcript reconstructions, with only a continuous output decoder. For the Colbrook-Hansen (CH23) singleton-window spectral/pseudospectral block, this raw preorder collapses the expected two-source structure: the diagonal exact spectral and fixed-\(\varepsilon\) pseudospectral sources are raw- and continuous-finite-query equivalent, and, for computable \(\varepsilon\) under the evaluation-name representations, TTE-finite-query equivalent, so the six-problem ambient is raw-principal. 
	We then introduce modal finite-query preorders, whose admissibility conditions may restrict encodings, decoders, reconstructions, uniformity, and geometric naturality. We also characterize TTE finite-query transport as computable point transport with a uniform finite interface trace; after forgetting the trace this gives strong Weihrauch reducibility, and the implication is strict. 
	
	Under a CH23 geometric modality generated by representation inclusions, unitary and graph relabelings, and neutral stabilizations, the same ambient has exactly two minimal exact sources. This gives a calibrated reformulation of the exact-basis problem: natural SCI families should be classified by modality-indexed exact bases and refinement maps, not by one raw preorder alone.
\end{abstract}

\maketitle
\begin{center}\small
\end{center}

\tableofcontents

\section{Introduction}
The Solvability Complexity Index (SCI) measures the number of successive limiting procedures required to solve a computational problem from finite information. In the raw type-\(G\) setting, an SCI computational problem is a tuple
\[
\mathcal P=(\Xi,\Omega,(\mathcal M,d),\Lambda),
\]
where \(\Omega\) is the instance set, \(\Lambda\) is the evaluation interface, and \(\Xi:\Omega\to \mathcal M\) is the target map; see \cite[Def.~2.4, 2.6, 2.10 and~2.19]{Sorg26Foundations}. The spectral SCI hierarchy was introduced and developed in connection with infinite-dimensional spectral computation, where sharp lower bounds express finite-information obstructions rather than algorithmic inefficiency, see \cite{hansen2011solvability,ben2015computing,ColbrookHansen22}.

For individual problems, exactness means $\SCIG(\mathcal P)=k$. For families, however, one must distinguish between the existence of a sharp witness, pointwise exactness of every member, and worst-case sharpness. The witness-sharpness framework of \cite{Sorg26Witness} addresses this distinction using finite-query evaluation reductions. In its raw form, a reduction $\mathcal S \leGfq \mathcal P$ consists of an encoding \(E\), a continuous decoder \(D\), and finite reconstructions of each target evaluation of \(E(A)\) from source evaluations of \(A\); see \cite[Def.~4.8]{Sorg26Witness}. Such reductions pull back type-\(G\) towers and hence preserve lower bounds, i.e.
\[
\mathcal S\leGfq \mathcal P \quad\Longrightarrow\quad \SCIG(\mathcal S) \le \SCIG(\mathcal P)
\]
by \cite[Lem.~4.9 and Thm.~4.10]{Sorg26Witness}. This naturally leads to an exact-basis problem: for a natural \(k\)-bounded ambient \(\mathcal U\), can the exact layer
\[
\mathcal O_k(\mathcal U)=\{\mathcal P\in \mathcal U : \SCIG(\mathcal P)=k\}
\]
be generated from a small set of canonical exact sources?

This paper shows that the raw finite-query preorder is not the final invariant for this question. The obstruction already appears in the CH23 singleton-window spectral
package. Fix a compact interval \(J\subset\mathbb R\) with non-empty interior and \(\varepsilon>0\). Let
\[
\mathcal P^\sigma_{J,\diag}
\]
be the diagonal singleton-window problem asking whether
\[
\sigma(A)\cap\{z\}=\varnothing
\]
for $z\in J$, and let
\[
\mathcal P^{\sigma_\varepsilon}_{J,\varepsilon,\diag}
\]
be the corresponding fixed-\(\varepsilon\) pseudospectral problem. These are singleton restrictions of the CH23 decision problems \(\Xi_3\) and \(\Xi_4\), whose sharp height
classification follows from \cite[Thm.~3.10 and Rem.~3.11]{ColbrookHansen22}. Geometrically, one expects exact spectrum and fixed-scale pseudospectrum to give different sources. Raw finite-query transport nevertheless collapses this distinction. We prove
\[
\mathcal P^\sigma_{J,\diag} \equiv_{\raw,\mathrm{fq}}
\mathcal P^{\sigma_\varepsilon}_{J,\varepsilon,\diag}.
\]
The proof constructs new diagonal target operators whose spectral accumulation patterns encode the source predicate. Consequently, the full six-problem CH23 ambient
\[
\mathcal U^\sharp_{2,J,\varepsilon} =
\{
\mathcal P^\sigma_{J,\diag}, \mathcal P^\sigma_{J,\gen}, \mathcal P^\sigma_{J,\Graph},
\mathcal P^{\sigma_\varepsilon}_{J,\varepsilon,\diag},
\mathcal P^{\sigma_\varepsilon}_{J,\varepsilon,\gen},
\mathcal P^{\sigma_\varepsilon}_{J,\varepsilon,\Graph}
\}
\]
is raw-principal, i.e.
\[
\operatorname{MinDeg}^{\raw}_2(\mathcal U^\sharp_{2,J,\varepsilon}) =
\{[\mathcal P^\sigma_{J,\diag}]_{\raw}\}.
\]
Thus the originally expected raw two-source theorem block is false.

The collapse is not merely caused by discontinuity. The same construction gives continuous finite-query transports and, under the natural evaluation-name representations, TTE-computable finite-query transports, i.e.
\[
\mathcal P^\sigma_{J,\diag} \equiv_{\cont,\mathrm{fq}}
\mathcal P^{\sigma_\varepsilon}_{J,\varepsilon,\diag},
\qquad
\mathcal P^\sigma_{J,\diag} \equiv_{\TTE,\mathrm{fq}}
\mathcal P^{\sigma_\varepsilon}_{J,\varepsilon,\diag}.
\]
Here the TTE assertion is made for computable \(\varepsilon\), or equivalently relative to \(\varepsilon\) if \(\varepsilon\) is treated as a fixed oracle parameter. $\TTE$ denotes a TTE-computable finite-query transport modality on SCI interfaces, not ordinary Weihrauch reducibility.  In \cref{subsec:TTEfqtransportsW} we prove a precise comparison theorem: TTE finite-query transport is equivalent to computable point transport with a uniform finite interface trace. Forgetting this trace yields a strong Weihrauch reduction between represented target maps, but the converse fails even when the represented target maps are strongly Weihrauch equivalent. This is the transport-side analogue of the uniformity issue in the comparison between SCI and Weihrauch theory, where raw type-\(G\) towers are not directly comparable with represented-space computability without a pure \(\mathcal R\)-\(\lim\) normal form; see \cite[Rem.~2.20, Def.~A.21 and Thm.~A.22]{Sorg26Foundations}. For background on represented spaces and Weihrauch reducibility, see e.g. \cite{weihrauch2000computable,BrattkaPauly2018,BrattkaGherardiPauly2021}.

We therefore introduce modal finite-query transports. A modality \(\mathfrak m\) specifies which encodings, decoders, finite transcript reconstructions, and uniformity or
naturality requirements are admissible. The raw, continuous, Borel, TTE, representation-preserving, and geometric preorders are examples. Under the CH23 geometric modality generated by representation inclusions, unitary relabelings, graph relabelings, and neutral stabilizations, the expected two-source structure is recovered:
\[
\operatorname{MinDeg}^{\geom^{\mathrm{CH23}}_{J,\varepsilon}}_2
(\mathcal U^\sharp_{2,J,\varepsilon}) =
\{
[\mathcal P^\sigma_{J,\diag}]_{\geom}, [\mathcal P^{\sigma_\varepsilon}_{J,\varepsilon,\diag}]_{\geom}
\}.
\]
Thus the same ambient is raw-principal but geometrically two-source. This is the central calibration result of the paper.

The final section reformulates the exact-basis problem accordingly. For raw-sound modalities
\[
\mathfrak m\preceq\raw,
\]
one obtains modal exact degrees, modal exact bases, and refinement maps
\[
\pi_{\mathfrak m\to\mathfrak n} : D^{\mathfrak m}_k(\mathcal U) \to D^{\mathfrak n}_k(\mathcal U),
\]
where $\mathfrak m\preceq\mathfrak n$. The corrected open problem to \cite[Open problem~1]{Sorg26Witness} is not to find one raw exact basis for every natural family, but to classify the modal degree profile
\[
\mathfrak m\longmapsto D^{\mathfrak m}_k(\mathcal U)
\]
and its compatible minimal-support data. Philosophically, the raw problem still has ordinary set-theoretic meaning, but computational and geometric meaning are modality-relative: TTE interprets the raw SCI object inside represented-space computability, while the geometric modality interprets it through natural operator-theoretic transformations. This framework-relative reading is compatible with Tarskian semantic truth, Carnap's distinction between internal questions and choice of framework, and Quine's warning against reducing meaning to a single equivalence relation, see \cite{Tarski1944,Carnap1950,quine2000two}.

\textbf{Structure of the paper.}
\Cref{sec:raw-layer} recalls the raw SCI layer and the original exact-basis boundary. \Cref{sec:CH23-test-case} introduces the CH23 singleton-window test case. \Cref{sec:raw-collapse} proves the raw collapse. \Cref{sec:modal-transports} develops modal finite-query transports and proves the trace comparison between TTE finite-query transport and strong Weihrauch reducibility. \Cref{sec:regularity-collapse} proves continuous and TTE finite-query collapse. \Cref{sec:modal-degree-machinery} develops modal exact-degree machinery. \Cref{sec:geometric-recovery} proves the geometric two-source theorem. \Cref{sec:corrected-open-problem} formulates the corrected modal exact-basis problem and gives an outlook.

\textbf{Notation.} Throughout, \(\mathbb N:=\{1,2,3,\dots\}\) and \(\mathbb N_0:=\{0,1,2,\dots\}\).

\section{The Raw SCI Layer And The Original Exact-Basis Boundary}\label{sec:raw-layer}
This section fixes the formal layer used in the rest of the paper.  The definitions are compatible with the foundational SCI framework of \cite{Sorg26Foundations} and the finite-query witness-sharpness framework of \cite{Sorg26Witness}. The main point is that the raw layer is extensional: it records finite dependence on evaluation data, but it does not by itself impose Type-2, Weihrauch, BSS, Borel, continuous, or geometric implementation requirements.

\subsection{SCI Computational Problems, General Algorithms, And Raw Type-G Towers}
\begin{definition}[SCI computational problem]\label{def:sci-problem}
	An SCI computational problem is a quadruple
	\[
	\mathcal P=(\Xi,\Omega,(\mathcal M,d),\Lambda),
	\]
	where \(\Omega\) is a set called the input class, \((\mathcal M,d)\) is an metric space called the output metric space, \(\Xi : \Omega\to \mathcal M\) is the target map, an \(\Lambda\) is an (complex-valued) evaluation interface. The usual consistency condition reads as
	\[
	\Xi(A)\ne \Xi(B) \quad\Longrightarrow\quad
	\exists f\in\Lambda\text{ such that } f(A)\ne f(B).
	\]
\end{definition}

Thus an SCI computational problem separates three pieces of data: the mathematical instances \(\Omega\), the information interface \(\Lambda\), and the target quantity \(\Xi\).  The consistency condition says that the interface is not allowed to identify two inputs with different target values. All later transport notions will act on this interface level.

The first algorithmic notion is deliberately weak. A general algorithm may choose its finite query set depending on the input, but once that finite transcript is fixed, both the output and the chosen query set must be determined by it.

\begin{definition}[General algorithm]\label{def:general-algorithm}
	Let \(\mathcal P=(\Xi,\Omega,(\mathcal M,d),\Lambda)\) be an SCI computational problem. A general algorithm for \(\mathcal P\) is a pair
	\[
	(\Gamma,\Lambda_\Gamma),
	\]
	where \(\Gamma:\Omega\to \mathcal M\) and \(\Lambda_\Gamma:\Omega\to[\Lambda]^{<\omega}\) assigns to each input a finite set of queried evaluations, such that for all \(A,B\in\Omega\)
	\begin{enumerate}[label=\textup{(G\arabic*)}]
		\item if \(f(B)=f(A)\) for every \(f\in\Lambda_\Gamma(A)\), then \(\Gamma(B)=\Gamma(A)\);
		\item if \(f(B)=f(A)\) for every \(f\in\Lambda_\Gamma(A)\), then \(\Lambda_\Gamma(B)=\Lambda_\Gamma(A)\).
	\end{enumerate}
\end{definition}

\begin{definition}[Raw type-G SCI]\label{def:raw-scig}
	Let \(\mathcal P=(\Xi,\Omega,(\mathcal M,d),\Lambda)\). A raw type-G tower of height \(0\) is a general algorithm \(\Gamma\) such that \(\Gamma(A)=\Xi(A)\) for all \(A\in\Omega\). For \(k\in \mathbb{N}\), a raw type-G tower of height \(k\) is a family of general algorithms
	\[
	\Gamma_{n_k,\ldots,n_1} : \Omega \to \mathcal M,
	\]
	where $n_1,\ldots,n_k)\in \bN^k$, such that, for every \(A\in\Omega\),
	\[
	\Xi(A) =
	\lim_{n_k\to\infty} \cdots \lim_{n_1\to\infty} \Gamma_{n_k,\ldots,n_1}(A)
	\]
	exists in \((\mathcal M,d)\). The raw type-G SCI is then defined by
	\[
	\SCIG(\mathcal P) := \min\{k\in\bN_0 : \mathcal P \text{ admits a raw type-G tower of height }k\},
	\]
	with value \(\infty\) if no such tower exists.
\end{definition}

The case \(k=0\) is exact finite-information solvability. Higher \(k\) allow successive semantic limits of such finite-information procedures. This is the raw type-\(G\) layer: it measures limiting depth, not computable implementability of the whole indexed tower.

\begin{remark}[Extensionality of the raw layer]\label{rem:raw-extensional}
	The raw type-G definitions impose finite-information dependence at the deepest algorithmic level, but they do not require the indexed tower table to be Type-2 computable, BSS-computable, Borel, continuous, or uniformly implemented. This is the source of the raw collapses below and precisely the extensionality issue emphasized in \cite[Rem.~2.20]{Sorg26Foundations}.
\end{remark}

\subsection{Raw Finite-Query Transport}
Finite-query transport is the reduction notion used to move lower bounds from one SCI computational problem to another. The target problem \(\mathcal P\) is allowed to be queried only through its interface, but each such target query must be reproducible from finitely many source queries.

\begin{definition}[Raw finite-query evaluation reduction]\label{def:raw-fq}
	Let
	\[
	\mathcal S=(\Psi,\Sigma,(\mathcal N,\rho),\Lambda_{\mathcal S}), \qquad
	\mathcal P=(\Xi,\Omega,(\mathcal M,d),\Lambda_{\mathcal P})
	\]
	be SCI computational problems. We write
	\[
	\mathcal S\le_{\raw,\fq} \mathcal P
	\]
	if there exist
	\begin{enumerate}[label=\textup{(F\arabic*)}]
		\item an encoding map
		\[
		E:\Sigma\to\Omega;
		\]
		\item a continuous decoder
		\[
		D:(\mathcal M,d) \to (\mathcal N,\rho);
		\]
		\item for every \(f\in\Lambda_{\mathcal P}\), \(m_f \in \mathbb{N}_0\), source evaluations
		\[
		\gamma_{f,1},\ldots,\gamma_{f,m_f}\in\Lambda_{\mathcal S},
		\]
		and a reconstruction map
		\[
		\vartheta_f: \operatorname{im}(\gamma_{f,1},\ldots,\gamma_{f,m_f}) \to \operatorname{im}(f),
		\]
		where for \(m_f=0\) the image of the set of source evaluations is the one-point set,
	\end{enumerate}
	such that for every \(A\in\Sigma\),
	\[
	D(\Xi(E(A)))=\Psi(A),
	\]
	and, for every \(f\in\Lambda_{\mathcal P}\),
	\[
	f(E(A)) = \vartheta_f\bigl( \gamma_{f,1}(A),\ldots,\gamma_{f,m_f}(A) \bigr).
	\]
	This is the zero-query-constant variant of the finite-query evaluation reduction of \cite[Def.~4.8]{Sorg26Witness}. In the notation of that paper this is \(\le_{\mathrm G,\fq}\). We write \(\le_{\raw,\fq}\) here to emphasize that this is the raw modality.
\end{definition}

Compared with \cite[Def.~4.8]{Sorg26Witness}, we allow \(m_f=0\) for constant reconstructions. This is equivalent to the original positive-query convention whenever
one fixes a dummy source evaluation; it is technically more convenient for the constant compact-window and zero off-diagonal reconstructions used below.

\begin{remark}[What is restricted in the raw preorder]\label{rem:raw-restrictions}
	In \cref{def:raw-fq}, only the decoder is required to be continuous. The encoding \(E\) is an arbitrary set-theoretic map, and the finite reconstruction maps \(\vartheta_f\) are arbitrary maps on finite transcript ranges. Thus raw finite-query transport is a finite-information coding relation, not a geometric or computability-theoretic reduction by itself.
\end{remark}

The distinction between the continuous decoder and the unrestricted encoding is central below. The decoder must commute with limits, which is why raw finite-query transport is
sound for type-\(G\) lower-bound transfer. The encoding, however, may still manufacture a new instance in a way that is not geometrically natural.

\subsection{Exact Layers And The Raw Exact-Basis Problem}
The raw exact-basis problem concerns the exact height-\(k\) part of an ambient class. The ambient \(\mathcal U\) should be thought of as a finite or structured package of related SCI computational problems, rather than a single problem.

\begin{definition}[Raw exact layer and raw exact degrees]\label{def:raw-exact-layer}
Let \(\mathcal U\) be a family of SCI computational problems, and let \(k\in\mathbb N\). The raw exact layer is
\[
\mathcal O_k(\mathcal U) := \{\mathcal P\in \mathcal U : \SCIG(\mathcal P)=k\}.
\]
For \(\mathcal P, \mathcal Q \in \mathcal O_k(\mathcal U)\), write
\[
\mathcal P\equiv_{\raw,\mathrm{fq}} \mathcal Q
\]
if
\[
\mathcal P\le_{\raw,\mathrm{fq}} \mathcal Q \quad\text{and}\quad \mathcal Q\le_{\raw,\mathrm{fq}} \mathcal P.
\]
The raw exact degree set is
\[
D^{\raw}_k(\mathcal U) := \mathcal O_k(\mathcal U) / {\equiv_{\raw,\mathrm{fq}}}.
\]
We write
\[
[\mathcal P]_{\raw}
\]
for the raw exact degree of \(\mathcal P\), and define
\[
[\mathcal P]_{\raw} \preceq [\mathcal Q]_{\raw} \quad:\Longleftrightarrow\quad \mathcal P\le_{\raw,\mathrm{fq}} \mathcal Q.
\]
\end{definition}

\begin{definition}[Raw minimal exact degrees]\label{def:raw-minimal-exact-degrees}
A degree $d\in D^{\raw}_k(\mathcal U)$ is raw-minimal if there is no $e\in D^{\raw}_k(\mathcal U)$ with $e\prec d$. The set of raw-minimal exact degrees is denoted by
\[
\operatorname{MinDeg}^{\raw}_k(\mathcal U).
\]
\end{definition}

If a single raw-minimal degree lies below every exact-height member of \(\mathcal U\), then \(\mathcal U\) is raw-principal. If several incomparable minimal degrees are needed, then \(\mathcal U\) has a non-principal exact basis. The CH23 example below shows that the raw answer can be coarser than the geometrically expected one.

\section{The CH23 Singleton-Window Spectral/Pseudospectral Test Case}\label{sec:CH23-test-case}
We now introduce the test case that exposes the limitation of the raw preorder. It is small enough to analyze completely, but it comes from the work in \cite{ColbrookHansen22}.

\subsection{The Diagonal Singleton-Window Sources}
We recall the diagonal class \(\Omega_D\), which is the diagonal subclass of the \(\ell^2(\mathbb N)\) operator class from \cite[Sec.~3.2.1]{ColbrookHansen22}.

\begin{definition}[Diagonal operator class]\label{def:diag-class}
	Let \(\Omega_{\diag}\) be the class of maximal closed diagonal operators on \(\ell^2(\mathbb N)\). Thus \(A\in\Omega_{\diag}\) means that there is a sequence
	\[
	a=(a_j)_{j\in\mathbb N}\subseteq \mathbb C
	\]
	such that
	\[
	\mathcal D(A) =
	\left\{
	x=(x_j)_{j\in\mathbb N}\in\ell^2(\mathbb N) : (a_jx_j)_{j\in\mathbb N}\in\ell^2(\mathbb N)
	\right\},
	\]
	and
	\[
	(Ax)_j=a_jx_j.
	\]
	Equivalently,
	\[
	Ae_j=a_je_j
	\]
	for $j\in\mathbb N$. The diagonal coefficient evaluations are
	\[
	\mu_j(A):=a_j.
	\]
	For compatibility with matrix-entry interfaces one may also include
	\[
	\mu_{ij}(A):=\langle Ae_j,e_i\rangle
	=
	\begin{cases}
		a_j,&i=j,\\
		0,&i\neq j.
	\end{cases}
	\]
\end{definition}

This maximal diagonal convention ensures that spectral accumulation of the diagonal entries is part of the spectrum. That accumulation phenomenon is exactly what the raw collapse construction will exploit.

\begin{definition}[Singleton windows]\label{def:singleton-windows}
	Fix a compact interval \(J\subset\bR\) with nonempty interior.  Let
	\[
	\cK_{\mathrm{sgl}}(J)
	:=
	\{\{z\}:z\in J\}.
	\]
	For each \(K=\{z\}\in\cK_{\mathrm{sgl}}(J)\), fix rational approximants
	\[
	r_n(K)\in\bQ+i\bQ, \qquad
	|r_n(K)-z|\le 2^{-(n+1)}.
	\]
	The compact-window evaluations are
	\[
	\rho_n(A,K):=r_n(K).
	\]
\end{definition}

The compact input is restricted to singletons only to isolate the sharpest local spectral-window question. The rational approximants \(r_n(K)\) are part of the represented finite information about the singleton window.

There are now two decision problems on the same diagonal instance space and with the same compressed interface. They differ only in the target: exact spectral intersection versus fixed-\(\varepsilon\) pseudospectral intersection.

\begin{definition}[Diagonal spectral and pseudospectral singleton problems]\label{def:diag-problems}
	The diagonal singleton exact spectral problem is
	\[
	\Psd :=
	(\Xi_\sigma,\Omega_{\diag}\times\cK_{\mathrm{sgl}}(J), (\{0,1\},d_{\mathrm{disc}}),\Lambda_{\diag}),
	\]
	where \(1\) means ``Yes'' and
	\[
	\Xi_\sigma(A,K) :=
	\begin{cases}
		1,&\Sp(A)\cap K=\varnothing,\\
		0,&\Sp(A)\cap K\ne\varnothing.
	\end{cases}
	\]
	For fixed \(\varepsilon>0\), the diagonal singleton fixed-\(\varepsilon\) pseudospectral problem is
	\[
	\Ppd :=
	(\Xi_{\sigma_\varepsilon},\Omega_{\diag}\times\cK_{\mathrm{sgl}}(J), (\{0,1\},d_{\mathrm{disc}}),\Lambda_{\diag}),
	\]
	where
	\[
	\Xi_{\sigma_\varepsilon}(A,K) :=
	\begin{cases}
		1,&\Sp_\varepsilon(A)\cap K=\varnothing,\\
		0,&\Sp_\varepsilon(A)\cap K\ne\varnothing.
	\end{cases}
	\]
	The diagonal interface used for these two diagonal singleton problems is the compressed interface
	\[
	\Lambda_{\diag} := \{\mu_j:j\in\mathbb N\} \cup \{\rho_n : n\in\mathbb N\},
	\]
	where
	\[
	\mu_j(A,K):=a_j \quad\text{if}\quad Ae_j=a_je_j,
	\]
	and
	\[
	\rho_n(A,K):=r_n(K).
	\]
	When a diagonal operator is transported into a general or graph CH23 representation, the additional target-side matrix-entry, support, or dispersion evaluations are
	reconstructed from this compressed diagonal interface by constants and diagonal coefficient queries. These are the singleton-window restrictions of the decision problems \(\Xi_3\) and \(\Xi_4\) from \cite[Thm.~3.10]{ColbrookHansen22}.
\end{definition}

The following elementary diagonal test translates both decision problems into distance conditions on the diagonal values. It is the local spectral calculation underlying all collapse and separation arguments in the paper.

Throughout the CH23 singleton-window block we use the closed pseudospectrum convention of \cite[§3, p.~9]{ColbrookHansen22},
\[
\sigma_\varepsilon(A) = \overline{\{z\in\mathbb C:\|(A-zI)^{-1}\|>\varepsilon^{-1}\}} .
\]
For comparison, the open matrix convention and its normal-matrix formula are discussed in \cite[§2, pp.~13-19, (2.1)-(2.17), Thm.~2.2]{trefethen2020spectra}, while
the operator definitions are given in \cite[§4, p.~31, (4.3)-(4.5), Thm.~4.3]{trefethen2020spectra}.

\begin{lemma}[Diagonal spectral and pseudospectral tests]\label{lem:diag-tests}
	Let \(A\in\Omega_{\diag}\), and write
	\[
	Ae_j=a_je_j.
	\]
	Let
	\[
	K=\{z\}\subseteq \mathbb C.
	\]
	Then
	\[
	\sigma(A)=\overline{\{a_j:j\in\mathbb N\}},
	\]
	and therefore
	\[
	\sigma(A)\cap K\neq\varnothing \quad\Longleftrightarrow\quad \inf_{j\in\mathbb N}|a_j-z|=0.
	\]
	Moreover, for every \(\varepsilon>0\),
	\[
	\sigma_\varepsilon(A)\cap K\neq\varnothing \quad\Longleftrightarrow\quad \inf_{j\in\mathbb N}|a_j-z|\le\varepsilon.
	\]
	Equivalently,
	\[
	\sigma_\varepsilon(A)\cap K=\varnothing \quad\Longleftrightarrow\quad \inf_{j\in\mathbb N}|a_j-z|>\varepsilon.
	\]
\end{lemma}

\begin{proof}
Let
\[
S:=\overline{\{a_j : j\in\mathbb N\}}.
\]
We first prove that \(\sigma(A)=S\). If \(\lambda\notin S\), then
\[
d_\lambda := \inf_{j\in\mathbb N}|a_j-\lambda|>0 .
\]
Define \(R_\lambda:\ell^2(\mathbb N)\to\ell^2(\mathbb N)\) by
\[
(R_\lambda y)_j=(a_j-\lambda)^{-1}y_j .
\]
Then \(R_\lambda\) is bounded and
\[
\|R_\lambda\|=\sup_{j\in\mathbb N}|a_j-\lambda|^{-1} = d_\lambda^{-1}.
\]
Moreover \(R_\lambda y\in \mathcal D(A)\), since
\[
a_j(a_j-\lambda)^{-1} = 1+\lambda(a_j-\lambda)^{-1}
\]
is uniformly bounded in \(j\). Hence \(R_\lambda=(A-\lambda I)^{-1}\), so \(\lambda\in\rho(A)\).

Conversely, let \(\lambda\in S\). If \(\lambda=a_j\) for some \(j\), then \(A-\lambda I\) has a non-trivial kernel, so \(\lambda\in\sigma(A)\). Otherwise there are indices \(j_n\) such that \(a_{j_n}\to\lambda\). Since \(\|e_{j_n}\|=1\) and
\[
\|(A-\lambda I)e_{j_n}\|=|a_{j_n}-\lambda|\to 0,
\]
the operator \(A-\lambda I\) cannot have a bounded inverse. Thus \(\lambda\in\sigma(A)\). This proves
\[
\sigma(A)=\overline{\{a_j : j\in\mathbb N\}}.
\]

For \(w\in\rho(A)\), the preceding computation gives
\[
\|(A-wI)^{-1}\| = \frac{1}{\inf_{j\in\mathbb N}|a_j-w|} = \frac{1}{\operatorname{dist}(w,\sigma(A))}.
\]
Therefore the open resolvent-level pseudospectrum is
\[
\{w:\operatorname{dist}(w,\sigma(A))<\varepsilon\},
\]
and the closed CH23 convention gives
\[
\sigma_\varepsilon(A) = \{w:\operatorname{dist}(w,\sigma(A))\le \varepsilon\}.
\]
For \(K=\{z\}\), this is exactly
\[
z\in\sigma_\varepsilon(A) \quad\Longleftrightarrow\quad \inf_{j\in\mathbb N}|a_j-z|\le \varepsilon .
\]
\end{proof}

Thus, on diagonal singleton inputs, exact spectral intersection is the condition that the diagonal values approach the singleton, while fixed-\(\varepsilon\) pseudospectral
intersection is the condition that they approach it within distance \(\varepsilon\). Raw finite-query transport can encode either threshold into a new diagonal accumulation pattern.

\subsection{The Six-Problem CH23 Ambient}
The diagonal problems are the sources. To obtain the full theorem-block ambient, we add the corresponding CH23 general \(\ell^2(\mathbb N)\) and graph representations.

\begin{definition}[The six-problem ambient]\label{def:six-ambient}
	Let
	\[
	\mathcal U^{\sigma,\mathrm{sgl}}_J := \{\Psd,\Psg,\Psr\},
	\]
	where \(\Psg\) and \(\Psr\) are the CH23 singleton-window exact spectral decision problems on the corresponding general \(\ell^2(\bN)\) and graph representation classes. Let
	\[
	\mathcal U^{\sigma_\varepsilon,\mathrm{sgl}}_{J,\varepsilon} := \{\Ppd,\Ppg,\Ppr\},
	\]
	where \(\Ppg\) and \(\Ppr\) are the corresponding fixed-\(\varepsilon\) pseudospectral problems. The six-problem ambient is then defined by
	\[
	\mathcal U^{\sharp}_{2,J,\varepsilon} := \mathcal U^{\sigma,\mathrm{sgl}}_J \cup \mathcal U^{\sigma_\varepsilon,\mathrm{sgl}}_{J,\varepsilon}.
	\]
	The intended block sources are
	\[
	\mathcal S^\sigma_J := \Psd, \qquad
	\mathcal S^{\sigma_\varepsilon}_{J,\varepsilon} := \Ppd.
	\]
\end{definition}

For the later representation embeddings, we recall only the pieces of the CH23 interfaces that must be finitely reconstructed from diagonal data.

\begin{definition}[General and graph singleton-window interfaces]\label{def:gen-graph-singleton-interfaces}
	The general \( \ell^2(\mathbb N)\) singleton-window problems are the CH23 problems on the bounded-dispersion \( \ell^2(\mathbb N)\)-operator class. Their evaluation interface contains
	\[
	m_{ij}(A,K):=\langle Ae_j,e_i\rangle,
	\]
	for $i,j\in\mathbb N$, the CH23 bounded-dispersion witness data, and the compact-window approximants
	\[
	\rho_n(A,K):=r_n(K).
	\]
	
	For the graph problem, fix a connected countable graph
	\[
	\mathcal G=(V,E_{\mathcal G}), \qquad V=\{v_1,v_2,\ldots\}.
	\]
	The CH23 graph interface contains coefficient evaluations
	\[
	\alpha_{ij}(A,K) := \alpha_A(v_i,v_j),
	\]
	local finite-support data, and the compact-window approximants \(\rho_n(A,K)\).
	
	The targets are respectively
	\[
	(A,K)\mapsto 1_{\{\sigma(A)\cap K=\varnothing\}}
	\]
	and
	\[
	(A,K)\mapsto 1_{\{\sigma_{\varepsilon}(A)\cap K=\varnothing\}}.
	\]
	All these problems use the same convention that \(1\) means ``Yes'' and \(0\) means ``No''.
\end{definition}

\begin{theorem}[Height input imported from CH23]\label{thm:height-input}
	For every $\mathcal P\in \mathcal U^{\sharp}_{2,J,\varepsilon}$, one has
	\[
	\SCIG(\mathcal P)=2.
	\]
\end{theorem}

\begin{proof}
	The six problems in $\mathcal U^{\sharp}_{2,J,\varepsilon}$	are the singleton-window restrictions of the CH23 decision problems
	\[
	\Xi_3(A,K)=1_{\{\sigma(A)\cap K=\varnothing\}}, \qquad
	\Xi_4(A,K)=1_{\{\sigma_\varepsilon(A)\cap K=\varnothing\}},
	\]
	on the diagonal, general \(\ell^2(\mathbb N)\), and graph domains.
	
	By \cite[Thm.~3.10]{ColbrookHansen22} for \(j=3,4\) the corresponding compact-intersection decision problems on the diagonal, general, and graph domains satisfy the sharp
	classification
	\[
	\notin \Delta^G_2 \qquad\text{and}\qquad \in \Pi^A_2.
	\]
	The paragraph following \cite[Thm.~3.10]{ColbrookHansen22} states that the lower bounds remain valid when the compact sets are restricted to a fixed compact subset of \(\mathbb R\), and by \cite[Rem.~3.11]{ColbrookHansen22} the same classification remains valid for singleton compact sets $K=\{z\}$.
	
	By the conventions in \cite[Def.~5.5-5.8]{ColbrookHansen22}, membership in $\Delta^G_{m+1}$ means general type-\(G\) SCI at most \(m\). Hence
	\[
	\mathcal P\notin\Delta^G_2
	\]
	implies
	\[
	\operatorname{SCI}_G(\mathcal P)>1.
	\]
	Thus every one of the six singleton-window problems has
	\[
	\operatorname{SCI}_G(\mathcal P)\ge2.
	\]
	
	On the other hand, the inclusion $\mathcal P\in\Pi^A_2$	gives an arithmetic tower of height at most \(2\), hence in particular a general type-\(G\) tower of height at most \(2\), since arithmetic towers are special cases of general towers. Therefore
	\[
	\operatorname{SCI}_G(\mathcal P)=2
	\]
	for every
	\[
	\mathcal P\in \mathcal U^{\sharp}_{2,J,\varepsilon}.
	\]
\end{proof}

Consequently the six-problem ambient is a level-\(2\) exact-basis test. Any collapse or separation below happens inside the same exact SCI layer.

\section{The Raw Collapse Obstruction}\label{sec:raw-collapse}
We now test the raw exact-basis expectation. The result is negative: the raw preorder identifies the two diagonal sources by allowing artificial diagonal encodings whose accumulation points store the source predicate.

\subsection{The Diagonal Spectral/Pseudospectral Raw Collapse}
The proof of this result uses two soft accumulation encodings. The first turns exact spectral intersection into membership of a fixed point in an \(\varepsilon\)-pseudospectrum.  The second turns fixed-\(\varepsilon\) pseudospectral intersection into exact spectral accumulation.

\begin{theorem}[Raw collapse of the diagonal singleton sources]\label{thm:raw-collapse}
	One has
	\[
	\Psd\equiv_{\raw,\fq}\Ppd .
	\]
	Moreover, the coordinate reconstruction maps in the proof can be chosen continuous; the collapse is not caused merely by discontinuous finite-transcript maps.
\end{theorem}

\begin{proof}
	We prove both raw finite-query reductions.
	
	Fix once and for all
	\[
	x_0\in J
	\]
	and a singleton
	\[
	L:=\{x_0\}.
	\]
	Fix also a bijection
	\[
	\nu:\mathbb N\to\mathbb N \times\mathbb N_0, \qquad r\mapsto(p(r),j(r)).
	\]
	This identifies
	\[
	\ell^2(\mathbb N) \cong \ell^2(\mathbb N \times\mathbb N_0)
	\]
	by a fixed unitary relabeling.
	
	\smallskip
	\noindent
	\textbf{Step 1: A raw transport from exact spectrum to fixed-\(\varepsilon\) pseudospectrum:}
	Let
	\[
	(A,K)\in\Omega_{\diag}(J), \qquad K=\{z\}, \qquad Ae_j=a_je_j.
	\]
	Define a diagonal operator
	\[
	B=B(A,K)
	\]
	on
	\[
	\ell^2(\mathbb N \times\mathbb N_0)
	\]
	by
	\[
	Be_{p,j}=b_{p,j} e_{p,j},
	\]
	where
	\[
	b_{p,j} := x_0+\varepsilon+|a_j-r_{p+2}(K)|+\frac1p.
	\]
	Equivalently, after the fixed relabeling \(\nu\), \(B\) is a diagonal operator on \(\ell^2(\mathbb N)\).
	
	This \(B\) is the maximal diagonal operator with diagonal entries \(b_{p,j}\). Hence it is closed and densely defined. Therefore
	\[
	B\in\Omega_{\diag}.
	\]
	Define
	\[
	E_{\sigma\to\varepsilon}(A,K):=(B(A,K),L)
	\]
	and let the decoder be
	\[
	D=\operatorname{id}_{\{0,1\}}.
	\]
	
	We prove the output identity. Put
	\[
	\delta:=\inf_j|a_j-z|.
	\]
	By \cref{lem:diag-tests},
	\[
	\Xi_\sigma(A,K)=0 \quad\Longleftrightarrow\quad \delta=0.
	\]
	
	Assume first that \(\delta=0\). For each \(p\), choose \(j_p\) such that
	\[
	|a_{j_p}-z|<\frac1p.
	\]
	Since
	\[
	|r_{p+2}(K)-z|\le2^{-(p+3)},
	\]
	we have
	\[
	|a_{j_p}-r_{p+2}(K)| \le |a_{j_p}-z|+|z-r_{p+2}(K)| < \frac1p+2^{-(p+3)} \longrightarrow 0.
	\]
	Hence
	\[
	b_{p,j_p} \longrightarrow x_0+\varepsilon.
	\]
	Thus
	\[
	x_0+\varepsilon\in\sigma(B).
	\]
	Consequently
	\[
	\operatorname{dist}(x_0,\sigma(B))\le\varepsilon.
	\]
	Since \(B\) is diagonal normal,
	\[
	x_0\in\sigma_\varepsilon(B).
	\]
	Therefore
	\[
	\Xi_{\sigma_\varepsilon}(B,L)=0=\Xi_\sigma(A,K).
	\]
	
	Assume next that \(\delta>0\). Choose \(p_0\) such that for every \(p\ge p_0\),
	\[
	|z-r_{p+2}(K)|<\frac{\delta}{2}.
	\]
	Then for \(p\ge p_0\) and every \(j\),
	\[
	|a_j-r_{p+2}(K)|\ge |a_j-z|-|z-r_{p+2}(K)| \ge \frac{\delta}{2}.
	\]
	Hence
	\[
	|b_{p,j}-x_0| = \varepsilon+|a_j-r_{p+2}(K)|+\frac1p \ge \varepsilon+\frac{\delta}{2}
	\]
	for $p\ge p_0$.	For the finitely many levels \(p<p_0\),
	\[
	|b_{p,j}-x_0| \ge \varepsilon+\frac1p \ge \varepsilon+\min_{1\le p<p_0}\frac1p
	\]
	if \(p_0>1\); if \(p_0=1\) there are no such levels. Thus there exists an $\eta>0$, such that
	\[
	|b_{p,j}-x_0|\ge\varepsilon+\eta.
	\]
	Therefore
	\[
	\operatorname{dist}(x_0,\sigma(B))>\varepsilon,
	\]
	so
	\[
	x_0\notin\sigma_\varepsilon(B).
	\]
	Thus
	\[
	\Xi_{\sigma_\varepsilon}(B,L)=1=\Xi_\sigma(A,K).
	\]
	
	We now verify finite-query reconstruction. A target diagonal coefficient at coordinate \(r\), with
	\[
	\nu(r)=(p,j),
	\]
	is
	\[
	b_{p,j} = x_0+\varepsilon+|\mu_j(A,K)-\rho_{p+2}(A,K)|+\frac1p,
	\]
	so it is reconstructed from the finite transcript
	\[
	(\mu_j,\rho_{p+2})
	\]
	by the map
	\[
	(u,v)\mapsto x_0+\varepsilon+|u-v|+\frac1p.
	\]
	The compact-window approximants of \(L\) are fixed constants. Hence every target evaluation in the compressed diagonal interface is finitely reconstructed from source evaluations and
	\[
	\mathcal P^\sigma_{J,\diag} \le_{\raw,\mathrm{fq}} \mathcal P^{\sigma_\varepsilon}_{J,\varepsilon,\diag}.
	\]
	
	\smallskip
	\noindent
	\textbf{Step 2: A raw transport from fixed-\(\varepsilon\) pseudospectrum to exact spectrum:}
	Let again
	\[
	(A,K)\in\Omega_{\diag}(J), \qquad K=\{z\}, \qquad Ae_j=a_je_j.
	\]
	Define a diagonal operator
	\[
	C=C(A,K)
	\]
	by
	\[
	Ce_{p,j}=c_{p,j}e_{p,j},
	\]
	where
	\[
	c_{p,j} := x_0+\max\{0,|a_j-r_{p+2}(K)|-\varepsilon\}+\frac1p.
	\]
	As above, \(C\) is a maximal closed diagonal operator. Hence
	\[
	C\in\Omega_{\diag}.
	\]
	Define
	\[
	E_{\varepsilon\to\sigma}(A,K) := (C(A,K),L),
	\]
	and again take
	\[
	D=\operatorname{id}_{\{0,1\}}.
	\]
	
	Put
	\[
	\delta:=\inf_j|a_j-z|.
	\]
	By \cref{lem:diag-tests},
	\[
	\Xi_{\sigma_\varepsilon}(A,K)=0 \quad\Longleftrightarrow\quad \delta\le\varepsilon.
	\]
	
	Assume first that \(\delta\le\varepsilon\). For every \(p\), choose \(j_p\) such that
	\[
	|a_{j_p}-z|<\varepsilon+\frac1p.
	\]
	Then
	\[
	|a_{j_p}-r_{p+2}(K)|-\varepsilon \le |a_{j_p}-z|+|z-r_{p+2}(K)|-\varepsilon < \frac1p+2^{-(p+3)}.
	\]
	Therefore
	\[
	0\le \max\{0,|a_{j_p}-r_{p+2}(K)|-\varepsilon\} < \frac1p+2^{-(p+3)}.
	\]
	Hence
	\[
	c_{p,j_p}\longrightarrow x_0.
	\]
	Thus
	\[
	x_0\in\sigma(C),
	\]
	and consequently
	\[
	\Xi_\sigma(C,L)=0=\Xi_{\sigma_\varepsilon}(A,K).
	\]
	
	Assume next that \(\delta>\varepsilon\). Choose \(p_0\) such that for all \(p\ge p_0\),
	\[
	|z-r_{p+2}(K)|<\frac{\delta-\varepsilon}{2}.
	\]
	Then for \(p\ge p_0\) and every \(j\),
	\[
	|a_j-r_{p+2}(K)|-\varepsilon \ge |a_j-z|-|z-r_{p+2}(K)|-\varepsilon \ge \frac{\delta-\varepsilon}{2}.
	\]
	Hence
	\[
	|c_{p,j}-x_0| \ge \frac{\delta-\varepsilon}{2}
	\]
	for $p\ge p_0$.	For the finitely many levels \(p<p_0\),
	\[
	|c_{p,j}-x_0|\ge \frac1p.
	\]
	Thus there is \(\eta>0\) such that
	\[
	|c_{p,j}-x_0|\ge\eta.
	\]
	Therefore
	\[
	x_0\notin\sigma(C),
	\]
	so
	\[
	\Xi_\sigma(C,L)=1=\Xi_{\sigma_\varepsilon}(A,K).
	\]
	
	The finite-query reconstruction is again explicit. A target diagonal coefficient at coordinate \(r\), with \(\nu(r)=(p,j)\), is
	\[
	c_{p,j} = x_0+\max\{0,|\mu_j(A,K)-\rho_{p+2}(A,K)|-\varepsilon\}+\frac1p.
	\]
	It is reconstructed from
	\[
	(\mu_j,\rho_{p+2})
	\]
	by the map
	\[
	(u,v)\mapsto x_0+\max\{0,|u-v|-\varepsilon\}+\frac1p.
	\]
	The compact-window target evaluations are the fixed approximants of \(L\). Hence every target evaluation in the compressed diagonal interface is finitely reconstructed from source evaluations and
	\[
	\mathcal P^{\sigma_\varepsilon}_{J,\varepsilon,\diag} \le_{\raw,\mathrm{fq}} \mathcal P^\sigma_{J,\diag}.
	\]
	
	Combining the two reductions gives
	\[
	\mathcal P^\sigma_{J,\diag} \equiv_{\raw,\mathrm{fq}} \mathcal P^{\sigma_\varepsilon}_{J,\varepsilon,\diag}.
	\]
\end{proof}

This equivalence is the basic obstruction. It does not say that the two spectral notions are geometrically the same. It says that the raw finite-query preorder is able to code one decision predicate into a newly manufactured diagonal spectral accumulation pattern.

\subsection{Raw Principalization Of The Six-Problem Ambient}
The diagonal collapse becomes a collapse of the full six-problem ambient because the diagonal representation embeds into the general and graph CH23 representations by finite-query transports.

By the CH23 definition of the dispersion data \cite[§3.2.1, (3.9)]{ColbrookHansen22}
\[
D_{f,n}(A)= \max\{\|(I-P_{f(n)})AP_n\|,\|(I-P_{f(n)})A^*P_n\|\},
\]
a diagonal operator satisfies \(D_{n,n}(A)=0\). Hence the target-side dispersion information may be reconstructed by the constant choices \(f(n)=n\) and \(c_n=2^{-n}\).

\begin{lemma}[Raw within-block representation embeddings]\label{lem:raw-within-block-representation-embeddings}
	One has
	\[
	\mathcal P^\sigma_{J,\diag} \le_{\raw,\mathrm{fq}} \mathcal P^\sigma_{J,\gen}, \qquad
	\mathcal P^\sigma_{J,\diag} \le_{\raw,\mathrm{fq}} \mathcal P^\sigma_{J,\Graph},
	\]
	and
	\[
	\mathcal P^{\sigma_\varepsilon}_{J,\varepsilon,\diag} \le_{\raw,\mathrm{fq}}
	\mathcal P^{\sigma_\varepsilon}_{J,\varepsilon,\gen}, \qquad
	\mathcal P^{\sigma_\varepsilon}_{J,\varepsilon,\diag} \le_{\raw,\mathrm{fq}} \mathcal P^{\sigma_\varepsilon}_{J,\varepsilon,\Graph}.
	\]
\end{lemma}

\begin{proof}
	We prove the spectral statements; the pseudospectral statements are identical, using unitary invariance of the \(\varepsilon\)-pseudospectrum.
	
	For diagonal-to-general, use the encoding
	\[
	E(A,K):=(A,K),
	\]
	where the same diagonal operator is viewed as an element of the CH23 general \(\ell^2(\mathbb N)\)-operator class. The decoder is
	\[
	D=\operatorname{id}_{\{0,1\}}.
	\]
	The target identity holds because the underlying operator and compact input are unchanged, i.e.
	\[
	\sigma(E(A,K))=\sigma(A).
	\]
	A target matrix-entry evaluation satisfies
	\[
	\langle Ae_j,e_i\rangle =
	\begin{cases}
		\mu_j(A,K),&i=j,\\
		0,&i\neq j.
	\end{cases}
	\]
	Thus each matrix entry is reconstructed from either one diagonal source evaluation or a constant. Compact-window approximants are passed through unchanged. For diagonal
	operators, the bounded-dispersion auxiliary data can be chosen canonically, for instance with
	\[
	f(n)=n, \qquad c_n=2^{-n},
	\]
	because
	\[
	(I-P_n)AP_n=0, \qquad (I-P_n)A^*P_n=0.
	\]
	These auxiliary target evaluations are therefore reconstructed by constants. Hence
	\[
	\mathcal P^\sigma_{J,\diag}\le_{\raw,\mathrm{fq}} \mathcal P^\sigma_{J,\gen}.
	\]
	
	For diagonal-to-graph, fix a connected countable graph
	\[
	\mathcal G=(V,E_{\mathcal G}), \qquad V=\{v_1,v_2,\ldots\},
	\]
	and the unitary
	\[
	U:\ell^2(\mathbb N)\to\ell^2(V), \qquad Ue_j=\mathbf 1_{v_j}.
	\]
	Encode then
	\[
	E(A,K):=(UAU^{-1},K).
	\]
	If
	\[
	Ae_j=a_je_j,
	\]
	then
	\[
	UAU^{-1}\mathbf 1_{v_j}=a_j\mathbf 1_{v_j}.
	\]
	Thus the graph coefficient is
	\[
	\alpha(v_i,v_j) =
	\begin{cases}
		a_j,&i=j,\\
		0,&i\neq j.
	\end{cases}
	\]
	Local support sets may be chosen as
	\[
	S_{v_j}:=\{v_j\}.
	\]
	Therefore all graph coefficient and local-support evaluations are reconstructed from diagonal source evaluations and constants. Compact-window approximants are unchanged. The target identity follows from
	\[
	\sigma(UAU^{-1})=\sigma(A).
	\]
	Hence
	\[
	\mathcal P^\sigma_{J,\diag}\le_{\raw,\mathrm{fq}} \mathcal P^\sigma_{J,\Graph}.
	\]
	For the fixed-\(\varepsilon\) pseudospectral block, the same encodings work because
	\[
	\sigma_\varepsilon(UAU^{-1}) = \sigma_\varepsilon(A),
	\]
	and because representation inclusion does not change the operator. This proves the two pseudospectral reductions.
\end{proof}

The block-stabilization perspective in this lemma is parallel to \cite[Prop.~5.14]{Sorg26Witness} as a two-sided finite-query transport for singleton-window block-diagonal stabilization is proved there.

\begin{corollary}[Raw principal collapse of the six-problem ambient]\label{cor:raw-principal}
	The six-problem ambient is raw-principal at height \(2\), i.e.
	\[
	\operatorname{MinDeg}^{\raw}_2(\mathcal U^{\sharp}_{2,J,\varepsilon}) = \{[\Psd]_{\raw}\} = \{[\Ppd]_{\raw}\}.
	\]
\end{corollary}

\begin{proof}
	By \cref{thm:height-input}, every member of $\mathcal U^\sharp_{2,J,\varepsilon}$ has exact raw height \(2\). By \cref{thm:raw-collapse},
	\[
	\mathcal P^\sigma_{J,\diag} \equiv_{\raw,\mathrm{fq}} \mathcal P^{\sigma_\varepsilon}_{J,\varepsilon,\diag}.
	\]
	By \cref{lem:raw-within-block-representation-embeddings},
	\[
	\mathcal P^\sigma_{J,\diag} \le_{\raw,\mathrm{fq}} \mathcal P^\sigma_{J,\gen}, \qquad
	\mathcal P^\sigma_{J,\diag} \le_{\raw,\mathrm{fq}} \mathcal P^\sigma_{J,\Graph},
	\]
	and
	\[
	\mathcal P^{\sigma_\varepsilon}_{J,\varepsilon,\diag} \le_{\raw,\mathrm{fq}}
	\mathcal P^{\sigma_\varepsilon}_{J,\varepsilon,\gen}, \qquad
	\mathcal P^{\sigma_\varepsilon}_{J,\varepsilon,\diag} \le_{\raw,\mathrm{fq}}
	\mathcal P^{\sigma_\varepsilon}_{J,\varepsilon,\Graph}.
	\]
	Using
	\[
	\mathcal P^\sigma_{J,\diag} \equiv_{\raw,\mathrm{fq}}
	\mathcal P^{\sigma_\varepsilon}_{J,\varepsilon,\diag},
	\]
	we conclude that
	\[
	\mathcal P^\sigma_{J,\diag} \le_{\raw,\mathrm{fq}} \mathcal P
	\]
	for $\mathcal P\in \mathcal U^\sharp_{2,J,\varepsilon}$. Now let
	\[
	[\mathcal Q]_{\raw} \in D^{\raw}_2(\mathcal U^\sharp_{2,J,\varepsilon})
	\]
	be raw-minimal. Since
	\[
	\mathcal P^\sigma_{J,\diag} \le_{\raw,\mathrm{fq}} \mathcal Q,
	\]
	we have
	\[
	[\mathcal P^\sigma_{J,\diag}]_{\raw} \preceq [\mathcal Q]_{\raw}.
	\]
	By minimality of \([\mathcal Q]_{\raw}\), this forces
	\[
	[\mathcal Q]_{\raw} = [\mathcal P^\sigma_{J,\diag}]_{\raw}.
	\]
	Hence
	\[
	\operatorname{MinDeg}^{\raw}_2(\mathcal U^\sharp_{2,J,\varepsilon}) = \{[\mathcal P^\sigma_{J,\diag}]_{\raw}\}.
	\]
	The equality
	\[
	[\mathcal P^\sigma_{J,\diag}]_{\raw} = [\mathcal P^{\sigma_\varepsilon}_{J,\varepsilon,\diag}]_{\raw}
	\]
	follows directly from \cref{thm:raw-collapse}.
\end{proof}

Thus the original raw two-source theorem-block picture cannot be correct. The next sections refine the preorder rather than changing the CH23 problems themselves.

\section{Modal Finite-Query Transports}\label{sec:modal-transports}
The raw collapse in \cref{thm:raw-collapse} shows that the reduction notion must remember more than finite codability. We therefore introduce modalities: admissibility conditions on encodings, decoders, transcript reconstructions, and uniformity or naturality requirements.

\subsection{Typed Interfaces And Modal Finite-Query Transport}
\begin{definition}[Typed SCI problem]\label{def:typed-problem}
	A typed SCI computational problem is a tuple
	\[
	\mathbf P=(\Xi_P,\Omega_P,\mathcal Y_P,\Lambda_P), \qquad \mathcal Y_P=(Y_P,d_P),
	\]
	where each evaluation \(\lambda\in\Lambda_P\) has a specified value space \(V_\lambda\) and is a map
	\[
	\lambda:\Omega_P\to V_\lambda.
	\]
	The untyped complex-valued case is obtained by taking \(V_\lambda=\bC\) for all \(\lambda\).
\end{definition}

The typed formulation is only of bookkeeping nature. It allows matrix entries, compact approximants, graph supports, natural-number data, and finite rational objects to be
handled by one transport definition.

Now the picture points to that a finite-query transport never sees the whole interface at once. It sees finite transcripts, so the modality must say which maps on such finite transcripts are allowed.

\begin{definition}[Finite transcript]\label{def:finite-transcript}
Let \(\mathbf S\) be a typed SCI computational problem and let
\[
\vec\gamma=(\gamma_1,\ldots,\gamma_m), \qquad m\in\mathbb N_0, \qquad \gamma_i\in\Lambda_S.
\]
If \(m\ge1\), the transcript map is
\[
\vec\gamma:\Omega_S\to V_{\gamma_1}\times\cdots\times V_{\gamma_m}, \qquad
x\mapsto(\gamma_1(x),\ldots,\gamma_m(x)).
\]
If \(m=0\), the transcript value space is the one-point space
\[
V_{\emptyset}:=\{\ast\},
\]
and
\[
\vec\gamma:\Omega_S\to\{\ast\}
\]
is the constant map. The transcript image is denoted by
\[
\operatorname{im}(\vec\gamma).
\]
A reconstruction of a target evaluation \(\lambda\in\Lambda_P\) from this transcript is a map
\[
\vartheta:\operatorname{im}(\vec\gamma)\to V_\lambda.
\]
\end{definition}

Here one might ask, if the zero-length transcript convention is necessary. To avoid complicated notation for constant reconstructions, such as fixed compact-window approximants or zero off-diagonal entries, this is a matter of notational convenience rather than logical necessity.

A modality specifies here three admissibility classes: encodings, decoders, and transcript reconstructions. Different choices produce different semantic worlds over the same raw
SCI computational problems as we will see in the following.

\begin{definition}[Finite-query transport modality]\label{def:modality}
	A finite-query transport modality \(\mathfrak m\) assigns, for every ordered pair of typed SCI problems \((\mathbf S, \mathbf P)\):
	\begin{enumerate}[label=\textup{(M\arabic*)}]
		\item an admissible encoding class
		\[
		\mathcal E_{\mathfrak m}(\mathbf S, \mathbf P) \subseteq \{E:\Omega_S\to\Omega_P\};
		\]
		\item an admissible decoder class
		\[
		\mathcal D_{\mathfrak m}(\mathbf P, \mathbf S) \subseteq \{D:Y_P\to Y_S\};
		\]
		\item for every \(\lambda\in\Lambda_P\) and every finite source transcript \(\vec\gamma\), a class of admissible reconstruction maps
		\[
		\Theta_{\mathfrak m}(\lambda;\vec\gamma) \subseteq \{\vartheta:\operatorname{im}(\vec\gamma)\to V_\lambda\}.
		\]
	\end{enumerate}
\end{definition}

For this definition to give a meaningful degree theory, admissible transports must compose. The next definition isolates the exact closure assumptions needed for this.

\begin{definition}[Modal finite-query transport]\label{def:modal-fq}
	Let \(\mathbf S, \mathbf P\) be typed SCI problems and let \(\mathfrak m\) be a finite-query transport modality. We write
	\[
	\mathbf S\lemfq{\mathfrak m} \mathbf P
	\]
	if there exist
	\[
	E\in\mathcal E_{\mathfrak m}(\mathbf S, \mathbf P),\qquad
	D\in\mathcal D_{\mathfrak m}(\mathbf P, \mathbf S),
	\]
	and, for every \(\lambda\in\Lambda_P\), a number
	\[
	m_\lambda\in\mathbb N_0,
	\]
	source evaluations
	\[
	\gamma_{\lambda,1},\ldots,\gamma_{\lambda,m_\lambda}\in\Lambda_S,
	\]
	and a reconstruction map
	\[
	\vartheta_\lambda \in \Theta_{\mathfrak m} \bigl(\lambda;(\gamma_{\lambda,1},\ldots,\gamma_{\lambda,m_\lambda})\bigr)
	\]
	such that for every \(x\in\Omega_S\),
	\[
	D\bigl(\Xi_P(E(x))\bigr)=\Xi_S(x),
	\]
	and for every \(\lambda\in\Lambda_P\),
	\[
	\lambda(E(x)) = \vartheta_\lambda \bigl( \gamma_{\lambda,1}(x),\ldots,\gamma_{\lambda,m_\lambda}(x) \bigr).
	\]
\end{definition}

\begin{definition}[Admissible finite-query transport modality]\label{def:admissible-modality}
	Let \(\mathfrak m\) be a finite-query transport modality in the sense of \cref{def:modal-fq}. We call \(\mathfrak m\) admissible if the following three conditions hold.
	
	\begin{enumerate}[label=\textup{(A\arabic*)}]
		\item \textbf{Identity data}
		For every typed SCI computational problem \(\mathbf P\),
		\[
		\operatorname{id}_{\Omega_P}\in\mathcal E_{\mathfrak m}(\mathbf P, \mathbf P),\qquad
		\operatorname{id}_{Y_P}\in\mathcal D_{\mathfrak m}(\mathbf P, \mathbf P).
		\]
		Moreover, for every \(\lambda\in\Lambda_P\), the identity reconstruction
		\[
		\vartheta_\lambda:\operatorname{im}(\lambda)\to V_\lambda, \qquad \vartheta_\lambda(t):=t,
		\]
		belongs to
		\[
		\Theta_{\mathfrak m}(\lambda;\lambda).
		\]
		
		\item \textbf{Closure of encodings and decoders}
		If
		\[
		E_{\mathbf S, \mathbf P}\in\mathcal E_{\mathfrak m}(\mathbf S, \mathbf P), \qquad
		E_{P,Q}\in\mathcal E_{\mathfrak m}(\mathbf P, \mathbf Q),
		\]
		then
		\[
		E_{P,Q}\circ E_{S,P}\in\mathcal E_{\mathfrak m}(\mathbf S, \mathbf Q).
		\]
		If
		\[
		D_{Q,P}\in\mathcal D_{\mathfrak m}(\mathbf Q, \mathbf P),\qquad
		D_{P,S}\in\mathcal D_{\mathfrak m}(\mathbf P, \mathbf S),
		\]
		then
		\[
		D_{P,S}\circ D_{Q,P}\in\mathcal D_{\mathfrak m}(\mathbf Q, \mathbf S).
		\]
		
		\item \textbf{Closure of finite transcript reconstructions}
		Let \(\mathbf S, \mathbf P, \mathbf Q\) be typed SCI computational problems and $\lambda\in\Lambda_Q$. Assume first that \(\lambda\) is reconstructed from a finite \(\mathbf P\)-transcript
		\[
		\vec\beta=(\beta_1,\ldots,\beta_r), \qquad \beta_i\in\Lambda_P,
		\]
		by an admissible reconstruction
		\[
		\varphi\in\Theta_{\mathfrak m}(\lambda;\vec\beta).
		\]
		Assume next that each \(\beta_i\) is reconstructed from a finite \(\mathbf S\)-transcript
		\[
		\vec\gamma_i=(\gamma_{i,1},\ldots,\gamma_{i,m_i}), \qquad \gamma_{i,j}\in\Lambda_S,
		\]
		by an admissible reconstruction
		\[
		\psi_i\in\Theta_{\mathfrak m}(\beta_i;\vec\gamma_i).
		\]
		The cases \(r=0\) and \(m_i=0\) are interpreted using the one-point transcript convention of \cref{def:finite-transcript}. Let
		\[
		\vec\gamma := (\gamma_{1,1},\ldots,\gamma_{1,m_1}, \gamma_{2,1},\ldots,\gamma_{r,m_r})
		\]
		be the concatenated \(\mathbf S\)-transcript. Suppose that the
		substitution map
		\[
		\Psi_{\vec\psi}:\operatorname{im}(\vec\gamma)\to
		V_{\beta_1}\times\cdots\times V_{\beta_r}
		\]
		defined by
		\[
		\Psi_{\vec\psi}
		(t_{1,1},\ldots,t_{r,m_r})
		:=
		\bigl(
		\psi_1(t_{1,1},\ldots,t_{1,m_1}),\ldots,
		\psi_r(t_{r,1},\ldots,t_{r,m_r})
		\bigr)
		\]
		has image contained in \(\operatorname{im}(\vec\beta)\).  Then
		\[
		\theta:=\phi\circ \Psi_{\vec\psi}
		\]
		belongs to
		\[
		\Theta_{\mathfrak m}(\lambda;\vec\gamma).
		\]
	\end{enumerate}
\end{definition}

\begin{proposition}[Modal finite-query transports form a preorder]\label{prop:modal-preorder}
	If \(\mathfrak m\) is admissible, then \(\le_{\mathfrak m,\fq}\) is reflexive and transitive.
\end{proposition}

\begin{proof}
	We prove reflexivity and transitivity separately.
	
	\smallskip
	\noindent\textbf{Reflexivity:}
	Let
	\[
	\mathbf P=(\Xi_P,\Omega_P,Y_P,\Lambda_P)
	\]
	be a typed SCI problem. By \textup{(A1)},
	\[
	\operatorname{id}_{\Omega_P}\in\mathcal E_{\mathfrak m}(\mathbf P, \mathbf P),\qquad
	\operatorname{id}_{Y_P}\in\mathcal D_{\mathfrak m}(\mathbf P, \mathbf P).
	\]
	For every \(\lambda\in\Lambda_P\), use the one-entry source transcript \((\lambda)\) and the reconstruction
	\[
	\vartheta_\lambda(t):=t.
	\]
	Again by \textup{(A1)}, this reconstruction is \(\mathfrak m\)-admissible. For every \(x\in\Omega_P\),
	\[
	\operatorname{id}_{Y_P}(\Xi_P(\operatorname{id}_{\Omega_P}(x))) = \Xi_P(x),
	\]
	and
	\[
	\lambda(\operatorname{id}_{\Omega_P}(x)) = \vartheta_\lambda(\lambda(x)).
	\]
	Hence
	\[
	\mathbf P\le_{\mathfrak m,\mathrm{fq}} \mathbf P.
	\]
	
	\smallskip
	\noindent \textbf{Transitivity:}
	Assume
	\[
	\mathbf S\le_{\mathfrak m,\mathrm{fq}} \mathbf P \qquad\text{and}\qquad \mathbf P\le_{\mathfrak m,\mathrm{fq}} \mathbf Q.
	\]
	Let the first transport be witnessed by
	\[
	E_{S,P}:\Omega_S\to\Omega_P, \qquad D_{P,S}:Y_P\to Y_S,
	\]
	and by reconstruction data for the evaluations of \(\mathbf P\). Let the second transport be witnessed by
	\[
	E_{P,Q}:\Omega_P\to\Omega_Q, \qquad
	D_{Q,P}:Y_Q\to Y_P,
	\]
	and by reconstruction data for the evaluations of \(\mathbf Q\).
	
	Define
	\[
	E_{S,Q}:=E_{P,Q}\circ E_{S,P}, \qquad
	D_{Q,S}:=D_{P,S}\circ D_{Q,P}.
	\]
	By \textup{(A2)},
	\[
	E_{S,Q}\in\mathcal E_{\mathfrak m}(\mathbf S, \mathbf Q), \qquad
	D_{Q,S}\in\mathcal D_{\mathfrak m}(\mathbf Q, \mathbf S).
	\]
	For \(x\in\Omega_S\), the output identity is
	\[
	\begin{aligned}
		D_{Q,S}\bigl(\Xi_Q(E_{S,Q}(x))\bigr)
		&= D_{P,S} \Bigl( D_{Q,P} \bigl( \Xi_Q(E_{P,Q}(E_{S,P}(x))) \bigr) \Bigr) \\
		&= D_{P,S} \bigl( \Xi_P(E_{S,P}(x)) \bigr) \\
		&= \Xi_S(x).
	\end{aligned}
	\]
	
	It remains to construct the finite reconstruction data. Fix
	\[
	\lambda\in\Lambda_Q.
	\]
	The second transport supplies a finite \(\mathbf P\)-transcript
	\[
	\vec\beta_\lambda=(\beta_1,\ldots,\beta_r), \qquad \beta_i\in\Lambda_P,
	\]
	and a reconstruction
	\[
	\varphi_\lambda\in\Theta_{\mathfrak m}(\lambda;\vec\beta_\lambda)
	\]
	such that, for every \(y\in\Omega_P\),
	\[
	\lambda(E_{P,Q}(y)) = \varphi_\lambda(\beta_1(y),\ldots,\beta_r(y)).
	\]
	For each \(i\), the first transport supplies a finite \(\mathbf S\)-transcript
	\[
	\vec\gamma_i=(\gamma_{i,1},\ldots,\gamma_{i,m_i})
	\]
	and a reconstruction
	\[
	\psi_i\in\Theta_{\mathfrak m}(\beta_i;\vec\gamma_i)
	\]
	such that, for every \(x\in\Omega_S\),
	\[
	\beta_i(E_{S,P}(x)) = \psi_i(\gamma_{i,1}(x),\ldots,\gamma_{i,m_i}(x)).
	\]
	Let \(\vec\gamma\) be the concatenation of all \(\vec\gamma_i\), and define $\theta_\lambda$ from \(\vec\gamma\) by the substitution formula in \textup{(A3)}. By \textup{(A3)},
	\[
	\theta_\lambda\in\Theta_{\mathfrak m}(\lambda;\vec\gamma).
	\]
	More further, for every \(x\in\Omega_S\),
	\[
	\begin{aligned}
		\lambda(E_{S,Q}(x))
		&= \lambda(E_{P,Q}(E_{S,P}(x)))\\
		&= \varphi_\lambda \bigl( \beta_1(E_{S,P}(x)),\ldots,\beta_r(E_{S,P}(x)) \bigr)\\
		&= \theta_\lambda(\vec\gamma(x)).
	\end{aligned}
	\]
	Thus every \(\mathbf Q\)-evaluation is finitely reconstructed from \(\mathbf S\)-evaluations by \(\mathfrak m\)-admissible data. Hence
	\[
	\mathbf S\le_{\mathfrak m,\mathrm{fq}} \mathbf Q.
	\]
\end{proof}

Thus every admissible modality gives a preorder and hence the potential for a degree theory. The raw preorder is only one point in this larger ordered family.

\subsection{The Raw Modality And The Modality Order}
The first modality we introduce is the one already used in the raw collapse theorem.

\begin{definition}[The raw modality]\label{def:raw-modality}
	The raw modality is
	\[
	\mathfrak m=\raw,
	\]
	where all set-theoretic encodings are admissible, all continuous output decoders are admissible, and all finite transcript reconstruction maps are admissible. Thus
	\[
	\mathbf S \lemfq{\raw} \mathbf P \quad\Longleftrightarrow\quad \mathbf S\leGfq \mathbf P.
	\]
\end{definition}

\begin{definition}[Order of modalities]\label{def:modality-order}
	For admissible modalities \(\mathfrak m\) and \(\mathfrak n\), write
	\[
	\mathfrak m\preceq\mathfrak n
	\]
	if for all SCI computational (typed) problems \(\mathbf P, \mathbf Q\),
	\[
	\mathbf P\lemfq{\mathfrak m} \mathbf Q \quad\Longrightarrow\quad \mathbf P\lemfq{\mathfrak n} \mathbf Q.
	\]
	Thus \(\mathfrak m\preceq\mathfrak n\) means that \(\mathfrak m\) is the stricter modality and \(\mathfrak n\) is the coarser modality.
\end{definition}

The order is contravariant to strength: stricter modalities have fewer admissible transports. Passing from a stricter modality to a coarser one can only identify more problems.

\subsection{Regularity-Controlled And Implemented Transport Modalities}
Borel transcript maps require a measurable structure on transcript images. The following trace convention avoids ambiguity when the image is not known to be a Borel subset of the ambient product.

\begin{definition}[Trace measurable transcript structure]\label{def:trace-transcript-structure}
	Assume that all evaluation value spaces \(V_\lambda\) are measurable spaces. Let
	\[
	\vec\gamma=(\gamma_1,\ldots,\gamma_m)
	\]
	for $\gamma_i\in\Lambda_S$ be a finite source transcript. Its value space is
	\[
	V_{\vec\gamma} := V_{\gamma_1}\times\cdots\times V_{\gamma_m}.
	\]
	The transcript image is
	\[
	\operatorname{im}(\vec\gamma) := \{\vec\gamma(x):x\in\Omega_S\} \subseteq V_{\vec\gamma}.
	\]
	We equip \(\operatorname{im}(\vec\gamma)\) with the trace sigma algebra
	\[
	\mathcal A_{\vec\gamma} :=
	\{
	B\cap \operatorname{im}(\vec\gamma): B\in\mathcal B(V_{\vec\gamma})
	\}.
	\]
	A reconstruction map
	\[
	\vartheta: \operatorname{im}(\vec\gamma)\to V_\lambda
	\]
	is called Borel if it is measurable as a map
	\[
	(\operatorname{im}(\vec\gamma),\mathcal A_{\vec\gamma}) \longrightarrow (V_\lambda,\mathcal B(V_\lambda)).
	\]
	This convention avoids the ambiguity of asking whether a map defined only on an arbitrary image subset is ``pointwise Borel''. We refer for standard Borel terminology and trace Borel structures to e.g. \cite[Ch.~II, Sec.~10-12]{Kechris1995}.
\end{definition}

We now list the regularity-controlled modalities used in this paper. The Borel variants are separated because a Borel decoder is not sound for raw type-\(G\) limit pullback in the same way that a continuous decoder is.

\begin{definition}[Regularity-controlled transport modalities]\label{def:regularity-controlled-modalities}
	Assume that all instance sets carry their evaluation topologies and evaluation sigma algebras, and that all output and evaluation value spaces carry their standard topological and measurable structures.
	
	\begin{enumerate}[label=\textup{(\roman*)}]
		\item The continuous modality $\cont$ is defined by
		\[
		\mathcal E_{\cont}(\mathbf S, \mathbf P) := C(\Omega_S,\Omega_P),
		\]
		\[
		\mathcal D_{\cont}(\mathbf P, \mathbf S) := C(Y_P,Y_S),
		\]
		and
		\[
		\Theta_{\cont}(\lambda;\vec\gamma) := C(\operatorname{im}(\vec\gamma),V_\lambda),
		\]
		where \(\operatorname{im}(\vec\gamma)\) carries the subspace topology inherited from \(V_{\vec\gamma}\).
		
		\item The Borel-with-continuous-decoder modality $\BorcD$ is defined by
		\[
		\mathcal E_{\BorcD}(\mathbf S, \mathbf P) := \{E:\Omega_S\to\Omega_P : E\text{ is Borel}\},
		\]
		\[
		\mathcal D_{\BorcD}(\mathbf P, \mathbf S) := C(Y_P,Y_S),
		\]
		and
		\[
		\Theta_{\BorcD}(\lambda;\vec\gamma) := \{\vartheta:\operatorname{im}(\vec\gamma)\to V_\lambda:
		\vartheta\text{ is Borel in the sense of \cref{def:trace-transcript-structure}}\}.
		\]
		This is the safest Borel refinement for type-\(G\) SCI, because the decoder remains continuous.
		
		\item The decoder-only Borel modality $\BorD$ is defined by
		\[
		\mathcal E_{\BorD}(\mathbf S, \mathbf P) := \{E:\Omega_S\to\Omega_P\},
		\]
		\[
		\mathcal D_{\BorD}(\mathbf P, \mathbf S) := \{D:Y_P\to Y_S:D\text{ is Borel}\},
		\]
		and
		\[
		\Theta_{\BorD}(\lambda;\vec\gamma) :=
		\{\text{all maps } \operatorname{im}(\vec\gamma)\to V_\lambda\}.
		\]
		This is the decoder-regular construction of \cite[Def.~4.13, Def.~4.14]{Sorg26Witness}; the preorder property is \cite[Prop.~4.15]{Sorg26Witness}, and the continuous-decoder instance is identified with \(\le_{G,\mathrm{fq}}\) in \cite[Rem.~4.16]{Sorg26Witness}.
		
		\item The full Borel modality $\Borfull$ is defined by requiring \(E\), \(D\), and all finite transcript reconstruction maps \(\vartheta_\lambda\) to be Borel. Thus
		\[
		\mathcal E_{\Borfull}(\mathbf S, \mathbf P) := \{E:\Omega_S\to\Omega_P:E\text{ is Borel}\},
		\]
		\[
		\mathcal D_{\Borfull}(\mathbf P, \mathbf S) := \{D:Y_P\to Y_S:D\text{ is Borel}\},
		\]
		and
		\[
		\Theta_{\Borfull}(\lambda;\vec\gamma) := \{\vartheta:\operatorname{im}(\vec\gamma)\to V_\lambda : \vartheta\text{ is Borel}\}.
		\]
		This modality is useful descriptively, but it is not automatically sound for ordinary type-\(G\) SCI, because a Borel decoder need not commute with limits.
		
		\item The TTE finite-query modality $\TTE$	is defined only after choosing represented-space structures for all instance spaces, output spaces, and evaluation value spaces, and after effectively indexing all evaluation families. A transport
		\[
		\mathbf S\le_{\TTE,\mathrm{fq}} \mathbf P
		\]
		requires
		
		\begin{enumerate}[label=\textup{(\alph*)}]
			\item \(E:\Omega_S\to\Omega_P\) has a computable realizer;
			
			\item \(D:Y_P\to Y_S\) has a computable realizer;
			
			\item uniformly in a code for each target evaluation \(\lambda\in\Lambda_P\), one can compute
			\[
			m_\lambda\in\mathbb N_0, \qquad \gamma_{\lambda,1},\ldots,\gamma_{\lambda,m_\lambda}\in\Lambda_S,
			\]
			and an index for a computable partial map
			\[
			\widehat\vartheta_\lambda : \subseteq V_{\gamma_{\lambda,1}}\times\cdots\times V_{\gamma_{\lambda,m_\lambda}} \to V_\lambda
			\]
			whose domain contains
			\[
			\operatorname{im}(\gamma_{\lambda,1},\ldots,\gamma_{\lambda,m_\lambda}),
			\]
			such that, for every \(x\in\Omega_S\),
			\[
			\lambda(E(x)) =
			\widehat\vartheta_\lambda \bigl( \gamma_{\lambda,1}(x),\ldots,\gamma_{\lambda,m_\lambda}(x) \bigr).
			\]
		\end{enumerate}
		The uniformity in \(\lambda\) is essential: without it, one has only a non-uniform family of pointwise finite-query simulations, not a Type-2 implementation of the transport. Here, represented spaces and realizers are used in the standard sense of computable analysis, see e.g. \cite[Def.~2.1, Def.~2.2]{BrattkaPauly2018}. The uniformity requirement in the target-evaluation index is the transport-side analogue of the uniformity requirement for implemented SCI towers in \cite[Def.~2.11, Rem.~2.20]{Sorg26Foundations}.
	\end{enumerate}
\end{definition}

The TTE item should be read as a modality for finite-query transports, not as a replacement definition of Weihrauch reducibility. More precisely we note

\subsection{TTE Finite-Query Transport And Strong Weihrauch Reducibility}\label{subsec:TTEfqtransportsW}

The TTE finite-query modality should not be identified with ordinary Weihrauch reducibility. The precise relation is the following: a TTE finite-query transport is exactly a strong Weihrauch reduction between the represented target maps whose Weihrauch preprocessor is induced by an SCI instance encoding and whose encoded target interface admits a uniformly computable finite trace through the source interface.

Throughout this subsection, let
\[
\mathbf S=(\Psi,\Omega_S,Y_S,\Lambda_S), \qquad
\mathbf P=(\Xi,\Omega_P,Y_P,\Lambda_P)
\]
be represented typed SCI computational problems. We assume that the instance spaces, output spaces, and all evaluation value spaces are represented spaces, and
that the evaluation families are effectively indexed, i.e.
\[
\Lambda_S=(\gamma_i)_{i\in I_S}, \qquad
\Lambda_P=(\lambda_e)_{e\in I_P}.
\]
For a target evaluation \(\lambda_e\), write \(V_e^P\) for its value space. For a source evaluation \(\gamma_i\), write \(V_i^S\) for its value space.
Let
\[
\widehat{\Psi}:\Omega_S\to Y_S, \qquad
\widehat{\Xi}:\Omega_P\to Y_P
\]
denote the represented target maps obtained from \(\Psi\) and \(\Xi\).

By ``effectively indexed'' we mean that \(I_S,I_P\subseteq\mathbb N\) are sets of valid evaluation codes and that the trace machines below are required to act correctly on valid target-evaluation indices. From such an index one knows the represented value space of the corresponding evaluation.

\begin{definition}[Uniform finite interface trace]
	Let
	\[
	E:\Omega_S\to\Omega_P
	\]
	be a computable map between the represented instance spaces. We say that \(E\) admits a \textit{uniform finite \(\Lambda_P\)-trace through \(\Lambda_S\)} if there is a computable procedure which, given an index \(e\in I_P\) for a target evaluation \(\lambda_e\), outputs
	\[
	m_e\in\mathbb N_0,\qquad i(e,1),\ldots,i(e,m_e)\in I_S,
	\]
	and an index for a computable partial map
	\[
	\widehat{\vartheta}_e: \subseteq V^S_{i(e,1)}\times\cdots\times V^S_{i(e,m_e)} \to V^P_e
	\]
	whose domain contains the finite transcript image
	\[
	\operatorname{im}\bigl( \gamma_{i(e,1)},\ldots,\gamma_{i(e,m_e)} \bigr),
	\]
	such that for every \(x\in\Omega_S\),
	\[
	\lambda_e(E(x)) = \widehat{\vartheta}_e \bigl( \gamma_{i(e,1)}(x),\ldots,\gamma_{i(e,m_e)}(x) \bigr).
	\]
	For \(m_e=0\), the product value space is the one-point represented space and the formula means that \(\lambda_e(E(x))\) is reconstructed by a computable constant.
\end{definition}

\begin{theorem}[Point-extensional trace characterization of TTE finite-query transport]\label{thm:tte-fq-weihrauch-trace}
	Let \(\mathbf S\) and \(\mathbf P\) be represented typed SCI problems as above. The following are equivalent.
	
	\begin{enumerate}
		\item[\textnormal{(i)}]
		\[
		\mathbf S \leq_{\mathrm{TTE},\mathrm{fq}} \mathbf P .
		\]
		
		\item[\textnormal{(ii)}]
		There exist computable point maps
		\[
		E:\Omega_S\to\Omega_P, \qquad D:Y_P\to Y_S,
		\]
		such that
		\[
		D(\Xi(E(x)))=\Psi(x)
		\]
		for $x\in\Omega_S$,	and such that \(E\) admits a uniform finite \(\Lambda_P\)-trace through \(\Lambda_S\).
	\end{enumerate}
	
	In either case, if \(\Phi_E\) is a computable realizer of \(E\) and \(\Phi_D\) is a computable realizer of \(D\), then
	\[
	\widehat{\Psi} \leq_{\mathrm{sW}} \widehat{\Xi}
	\]
	is witnessed by the strong Weihrauch preprocessor \(\Phi_E\) and postprocessor \(\Phi_D\). Explicitly, for every realizer \(G\vdash\widehat{\Xi}\),
	\[
	\Phi_D\circ G\circ \Phi_E
	\]
	is a realizer of \(\widehat{\Psi}\).
\end{theorem}

\begin{proof}
	Assume first that \(\mathbf S \leq_{\mathrm{TTE},\mathrm{fq}} \mathbf P\). By definition of the TTE finite-query modality, the transport is witnessed by a computably realized
	encoding
	\[
	E:\Omega_S\to\Omega_P,
	\]
	a computably realized decoder
	\[
	D:Y_P\to Y_S,
	\]
	and, uniformly in an index \(e\) for each target evaluation \(\lambda_e\in\Lambda_P\), a finite list of source evaluations
	\[
	\gamma_{i(e,1)},\ldots,\gamma_{i(e,m_e)}
	\]
	together with an index for a computable partial reconstruction map
	\[
	\widehat{\vartheta}_e: \subseteq V^S_{i(e,1)}\times\cdots\times V^S_{i(e,m_e)} \to V^P_e
	\]
	such that
	\[
	\lambda_e(E(x)) = \widehat{\vartheta}_e \bigl( \gamma_{i(e,1)}(x),\ldots,\gamma_{i(e,m_e)}(x) \bigr)
	\]
	for all \(x\in\Omega_S\). This is exactly a uniform finite \(\Lambda_P\)-trace through \(\Lambda_S\). The output identity in the definition of \(\mathbf S \leq_{\mathrm{TTE},\mathrm{fq}} \mathbf P\) is precisely
	\[
	D(\Xi(E(x)))=\Psi(x).
	\]
	Hence (ii) holds.
	
	Conversely, assume (ii). The computable maps \(E\) and \(D\) supply the TTE-admissible encoding and decoder. The uniform finite trace supplies, uniformly in every target evaluation index \(e\), the finite source transcript and computable partial reconstruction map required in the definition of \(\leq_{\mathrm{TTE},\mathrm{fq}}\). The output identity gives the target identity condition. Therefore
	\[
	\mathbf S\leq_{\mathrm{TTE},\mathrm{fq}} \mathbf P.
	\]
	
	It remains only to verify the Weihrauch statement. Let \(p\) be a name of \(x\in\Omega_S\). Since \(\Phi_E\) realizes \(E\), the name \(\Phi_E(p)\) is a name of \(E(x)\). If \(G\vdash\widehat{\Xi}\), then \(G(\Phi_E(p))\) is a name of \(\Xi(E(x))\). Since \(\Phi_D\) realizes \(D\), the value
	\[
	\Phi_D(G(\Phi_E(p)))
	\]
	is a name of
	\[
	D(\Xi(E(x)))=\Psi(x).
	\]
	Thus \(\Phi_D\circ G\circ \Phi_E\) realizes \(\widehat{\Psi}\). The postprocessor \(\Phi_D\) does not use the original source name \(p\), so the reduction is strong Weihrauch reducibility.
\end{proof}

\begin{corollary}[The forgetful implication is not reversible]\label{cor:sW-not-tte-fq}
	The implication
	\[
	\mathbf S\leq_{\mathrm{TTE},\mathrm{fq}} \mathbf P \quad\Longrightarrow\quad \widehat{\Psi}\leq_{\mathrm{sW}}\widehat{\Xi}
	\]
	is strict. In fact, there are represented SCI problems \(\mathbf S\) and \(\mathbf P\) such that
	\[
	\widehat{\Psi}\equiv_{\mathrm{sW}} \widehat{\Xi},
	\]
	but
	\[
	\mathbf S \nleq_{\mathrm{raw},\mathrm{fq}} \mathbf P,
	\]
	and hence also
	\[
	\mathbf S \nleq_{\mathrm{TTE},\mathrm{fq}} \mathbf P.
	\]
\end{corollary}

\begin{proof}
	Let
	\[
	\mathcal C:=\{0,1\}^{\mathbb N_0}
	\]
	be the Cantor space with its standard representation. For \(n\in\mathbb N_0\), let
	\[
	\pi_n:\mathcal C\to\{0,1\}, \qquad \pi_n(x):=x(n),
	\]
	be the coordinate projections. Define also the computable injective map
	\[
	\tau:\mathcal C\to[0,1], \qquad
	\tau(x):=\sum_{n=0}^{\infty} \frac{2x(n)}{3^{n+1}}.
	\]
	The map \(\tau\) is injective: if \(x\neq y\) and \(n_0\) is the least index with \(x(n_0)\neq y(n_0)\), then
	\[
	|\tau(x)-\tau(y)| \geq \frac{2}{3^{n_0+1}} - \sum_{n>n_0}\frac{2}{3^{n+1}} = \frac{1}{3^{n_0+1}} >0.
	\]
	
	Define the two represented typed SCI problems
	\[
	\mathbf S := \bigl( \operatorname{id}_{\mathcal C}, \mathcal C, \mathcal C, \{\pi_n : n\in\mathbb N_0\} \bigr)
	\]
	and
	\[
	\mathbf P := \bigl( \operatorname{id}_{\mathcal C}, \mathcal C, \mathcal C, \{\pi_n : n\in\mathbb N_0\}\cup\{\tau\} \bigr).
	\]
	The value space of each \(\pi_n\) is the discrete represented space \(\{0,1\}\), and the value space of \(\tau\) is \([0,1]\) with its usual Cauchy representation.
	
	After forgetting the SCI interfaces, both represented target maps are just
	\[
	\operatorname{id}_{\mathcal C} : \mathcal C\to\mathcal C.
	\]
	Hence
	\[
	\widehat{\Psi}\equiv_{\mathrm{sW}} \widehat{\Xi}
	\]
	is witnessed in both directions by the identity preprocessor and identity postprocessor.
	
	We now show that nevertheless
	\[
	\mathbf S \nleq_{\mathrm{raw},\mathrm{fq}} \mathbf P.
	\]
	Suppose, toward a contradiction, that there is a raw finite-query transport from \(\mathbf S\) to \(\mathbf P\). Let
	\[
	E:\mathcal C\to\mathcal C, \qquad D:\mathcal C\to\mathcal C
	\]
	be its encoding and decoder. Since the two target maps are identities, the output identity of the transport says
	\[
	D(E(x))=x,
	\]
	where $x\in\mathcal C$.	In particular, \(E\) is injective.
	
	Now apply the finite-query reconstruction requirement to the target evaluation
	\[
	\tau\in\Lambda_P.
	\]
	There must be finitely many source evaluations
	\[
	\pi_{n_1},\ldots,\pi_{n_m}
	\]
	and a reconstruction map
	\[
	\vartheta: \operatorname{im}(\pi_{n_1},\ldots,\pi_{n_m}) \to[0,1]
	\]
	such that for every \(x\in\mathcal C\),
	\[
	\tau(E(x)) = \vartheta(x(n_1),\ldots,x(n_m)).
	\]
	Choose distinct \(x,x'\in\mathcal C\) such that
	\[
	x(n_j)=x'(n_j)
	\]
	for $j=1,\ldots,m$.	Then the preceding identity gives
	\[
	\tau(E(x))=\tau(E(x')).
	\]
	Since \(\tau\) is injective, this implies
	\[
	E(x)=E(x').
	\]
	Applying \(D\) and using \(D\circ E=\operatorname{id}_{\mathcal C}\), we get
	\[
	x=D(E(x))=D(E(x'))=x',
	\]
	contradicting the choice of \(x\neq x'\). Hence no raw finite-query transport exists. Since every TTE finite-query transport is, in particular, a raw finite-query transport after forgetting effectivity,
	\[
	\mathbf S \nleq_{\mathrm{TTE},\mathrm{fq}} \mathbf P.
	\]
\end{proof}

\begin{remark}
	The obstruction in \cref{cor:sW-not-tte-fq} is \textit{not} computability of the represented target maps: the represented target maps are identical. The obstruction is the extra SCI-interface requirement. Strong Weihrauch reducibility sees only the target maps \(\widehat{\Psi}\) and \(\widehat{\Xi}\). TTE finite-query transport also sees the target evaluation interface \(\Lambda_P\) and requires every evaluation of the encoded target instance \(E(x)\) to be uniformly reconstructible from finitely many source evaluations of \(x\).
	
	Ordinary Weihrauch reducibility is still further away from the finite-query transport notion. Its postprocessor may use the original source name. To model ordinary Weihrauch reducibility on the SCI-transport side, one would need a different transport modality with an input-dependent decoder, for example a computably realized map
	\[
	D:\Omega_S\times Y_P\to Y_S,
	\]
	or equivalently a name-level postprocessor of the form
	\[
	K(p,G(H(p))).
	\]
	The finite-query transports used in this paper use the strong, input-free decoder form because this is the form compatible with pullback of raw type-G towers.
\end{remark}

\subsection{Representation-Preserving And Geometric Refinements}
Regularity and naturality are independent axes. A map may be computable but geometrically artificial, or geometrically natural but based on noncomputable fixed data. The following refinement construction lets us impose naturality on top of any base regularity modality.

\begin{definition}[Representation-preserving and geometric refinements of a base modality]\label{def:rep-geom-refinements}
	Let \(\mathfrak a\) be a base modality such as
	\[
	\raw,\qquad \cont,\qquad \BorcD,\qquad \TTE .
	\]
	Assume that for every pair of typed SCI computational problems \(\mathbf S, \mathbf P\) we have specified two classes of encodings
	\[
	\mathcal R(\mathbf S, \mathbf P)\subseteq \mathcal G(\mathbf S, \mathbf P) \subseteq \{E:\Omega_S\to\Omega_P\},
	\]
	where
	\begin{itemize}
		\item \(\mathcal R(\mathbf S, \mathbf P)\) is the class of representation-preserving encodings;
		\item \(\mathcal G(\mathbf S, \mathbf P)\) is the class of geometric encodings.
	\end{itemize}
	We assume that \(\mathcal R\) and \(\mathcal G\) contain identities and are closed under composition.
	
	The representation-preserving refinement of \(\mathfrak a\) is the modality $\rep^{\mathfrak a}$ defined by
	\[
	\mathcal E_{\rep^{\mathfrak a}}(\mathbf S, \mathbf P) := \mathcal E_{\mathfrak a}(\mathbf S, \mathbf P)\cap\mathcal R(\mathbf S, \mathbf P),
	\]
	\[
	\mathcal D_{\rep^{\mathfrak a}}(\mathbf P, \mathbf S) := \mathcal D_{\mathfrak a}(\mathbf P, \mathbf S),
	\]
	\[
	\Theta_{\rep^{\mathfrak a}} := \Theta_{\mathfrak a}.
	\]
	Similarly, the geometric refinement of \(\mathfrak a\) is the modality $\geom^{\mathfrak a}$ defined by
	\[
	\mathcal E_{\geom^{\mathfrak a}}(\mathbf S, \mathbf P) :=
	\mathcal E_{\mathfrak a}(\mathbf S, \mathbf P)\cap\mathcal G(\mathbf S, \mathbf P),
	\]
	\[
	\mathcal D_{\geom^{\mathfrak a}}(\mathbf P, \mathbf S) := \mathcal D_{\mathfrak a}(\mathbf P, \mathbf S),
	\]
	\[
	\Theta_{\geom^{\mathfrak a}} := \Theta_{\mathfrak a}.
	\]
	When \(\mathfrak a=\raw\), we abbreviate
	\[
	\rep:=\rep^{\raw}, \qquad \geom:=\geom^{\raw}.
	\]
\end{definition}

The bare symbols \(\rep\) and \(\geom\) mean representation-preserving and geometric refinements of the raw modality. Later, the CH23 geometric modality will be a concrete choice of the abstract class \(\mathcal G(\mathbf S, \mathbf P)\).

\subsection{The Modality Poset}
We first check that the calibrated modalities really satisfy the abstract admissibility axioms.

\begin{proposition}[Admissibility of the calibrated modalities]\label{prop:calibrated-modalities-admissible}
	Under the closure assumptions stated above, the modalities
	\[
	\raw,\quad \cont,\quad \BorcD,\quad \BorD,\quad \Borfull,\quad \TTE,
	\]
	and their representation-preserving and geometric refinements
	\[
	\rep^{\mathfrak a}, \qquad \geom^{\mathfrak a},
	\]
	are admissible finite-query transport modalities.
\end{proposition}

\begin{proof}
	We verify the three conditions of \cref{def:admissible-modality}.
	
	\smallskip
	\noindent \textbf{The raw modality:}
	For \(\raw\), every encoding is admissible, every finite transcript reconstruction is admissible, and the decoder class is the class of continuous maps between output metric
	spaces. Identity encodings and identity reconstructions are therefore admissible. The identity decoder is continuous, and compositions of continuous decoders are continuous.
	Since all reconstruction maps are allowed, the composite reconstruction in \textup{(A3)} is automatically raw-admissible. Hence \(\raw\) is admissible.
	
	\smallskip
	\noindent \textbf{The continuous modality:}
	For \(\cont\), identity encodings, decoders, and transcript reconstructions are continuous. Compositions of continuous encodings and decoders are continuous. For reconstructions suppose
	\[
	\varphi:\operatorname{im}(\vec\beta)\to V_\lambda
	\]
	is continuous and each
	\[
	\psi_i:\operatorname{im}(\vec\gamma_i)\to V_{\beta_i}
	\]
	is continuous. The product map
	\[
	\Psi: \operatorname{im}(\vec\gamma) \to \operatorname{im}(\vec\beta)
	\]
	defined by
	\[
	\Psi(t_{1,1},\ldots,t_{r,m_r}) := (\psi_1(t_{1,1},\ldots,t_{1,m_1}),\ldots, \psi_r(t_{r,1},\ldots,t_{r,m_r}))
	\]
	is continuous for the subspace product topologies. Hence
	\[
	\theta:=\varphi\circ\Psi
	\]
	is continuous. Thus \(\cont\) is admissible.
	
	\smallskip
	\noindent \textbf{The Borel-with-continuous-decoder modality:}
	For	\(\Bor_{E,\vartheta;D=\mathrm{cont}}\),	identity encodings and identity reconstructions are Borel, and identity decoders are continuous. Compositions of Borel encodings are Borel, and compositions of continuous decoders are continuous.
	
	For transcript reconstructions, use the trace measurable structure from \cref{def:trace-transcript-structure}. If
	\[
	\varphi:\operatorname{im}(\vec\beta)\to V_\lambda
	\]
	is Borel and each
	\[
	\psi_i:\operatorname{im}(\vec\gamma_i)\to V_{\beta_i}
	\]
	is Borel, then the product map
	\[
	\Psi: \operatorname{im}(\vec\gamma) \to \operatorname{im}(\vec\beta)
	\]
	defined as in the $\cont$ case is measurable with respect to the trace sigma algebras. Indeed, for each coordinate projection \(\pi_i\) on
	\[
	V_{\beta_1}\times\cdots\times V_{\beta_r},
	\]
	the coordinate map
	\[
	\pi_i\circ\Psi=\psi_i\circ\pi_{\vec\gamma_i}
	\]
	is measurable. Since the product sigma algebra is generated by coordinate cylinders, \(\Psi\) is measurable. Hence
	\[
	\theta=\varphi\circ\Psi
	\]
	is Borel. Thus \(\Bor_{E,\vartheta;D=\mathrm{cont}}\) is admissible.
	
	\smallskip
	\noindent \textbf{The decoder-only Borel modality:}
	For \(\Bor_D\), encodings and transcript reconstructions are unrestricted. The decoder class consists of Borel maps. Identity maps are Borel, and compositions of Borel maps
	are Borel. Since reconstructions are unrestricted, the composite reconstruction is automatically admissible. Hence \(\Bor_D\) is admissible.
	
	\smallskip
	\noindent \textbf{The full Borel modality:}
	For \(\Bor_{\mathrm{full}}\), encodings, decoders, and transcript reconstructions are all Borel. Identity maps are Borel and compositions of Borel maps are Borel. The transcript
	part is handled exactly as in the \(\Bor_{E,\vartheta;D=\mathrm{cont}}\) case. Hence \(\Bor_{\mathrm{full}}\) is admissible.
	
	\smallskip
	\noindent \textbf{The TTE modality:}
	For \(\TTE\), identity maps on represented spaces have computable realizers, and compositions of computably realized maps have computable realizers. Hence identity and
	composition closure hold for encodings and decoders.
	
	For reconstructions, suppose a target evaluation \(\lambda\) of \(\mathbf Q\) is uniformly reconstructed from \(\mathbf P\)-evaluations
	\[
	\beta_1,\ldots,\beta_r
	\]
	by a computable partial map \(\widehat\varphi\). Suppose each \(\beta_i\) is uniformly reconstructed from \(\mathbf S\)-evaluations
	\[
	\gamma_{i,1},\ldots,\gamma_{i,m_i}
	\]
	by a computable partial map \(\widehat\psi_i\). From an index for \(\lambda\), the TTE data compute the list of \(\beta_i\)'s and an index for \(\widehat\varphi\). From each
	index for \(\beta_i\), the TTE data compute the list of \(\gamma_{i,j}\)'s and an index for \(\widehat\psi_i\). By effective pairing, finite list concatenation, and effective composition of partial computable maps on represented spaces, one computes uniformly an index for
	\[
	\widehat\theta := \widehat\varphi\circ (\widehat\psi_1,\ldots,\widehat\psi_r).
	\]
	Its domain contains the relevant transcript image, and on that image it gives the required composite reconstruction. Therefore the TTE modality is admissible.
	
	\smallskip
	\noindent \textbf{Representation-preserving and geometric refinements:}
	Let \(\mathfrak a\) be one of the admissible base modalities. By assumption, the representation-preserving encoding classes \(\mathcal R(\mathbf S, \mathbf P)\) contain identities and are closed under composition; the same holds for the geometric classes \(\mathcal G(\mathbf S, \mathbf P)\). Therefore intersecting the encoding class of \(\mathfrak a\) with \(\mathcal R\) or with \(\mathcal G\) preserves identity and composition closure. The decoder and reconstruction classes are inherited from \(\mathfrak a\), already shown admissible. Hence \(\rep^{\mathfrak a}\) and \(\geom^{\mathfrak a}\) are admissible.
\end{proof}

The next theorem records the basic order relations among the modalities. It is a partial order, not a single hierarchy: computable regularity and geometric naturality constrain different aspects of a transport.

For the proof of the theorem we shortly recall that computable realizability implies represented-space continuity, see \cite[Def.~2.1 and the discussion after it]{BrattkaPauly2018}.

\begin{theorem}[Basic modality-order calibration]\label{thm:basic-modality-order-calibration}
	Assume that the topologies used in the continuous modality are the represented-space topologies induced by the chosen representations, or at least that every computably
	realized map in the TTE modality is continuous for the topologies used in the continuous modality. Then the following modality-order relations hold.
	
	\begin{enumerate}[label=\textup{(\alph*)}]
		\item
		\[
		\TTE \preceq \cont \preceq \BorcD \preceq \raw .
		\]
		
		\item
		\[
		\raw \preceq \BorD .
		\]
		
		\item For every base modality
		\[
		\mathfrak a\in \{\raw,\cont,\BorcD,\TTE\},
		\]
		one has
		\[
		\rep^{\mathfrak a} \preceq \geom^{\mathfrak a} \preceq \mathfrak a .
		\]
		In particular,
		\[
		\RepTTE \preceq \GeomTTE \preceq \TTE \preceq \cont \preceq \BorcD \preceq \raw,
		\]
		and
		\[
		\RepCont \preceq \GeomCont \preceq \cont \preceq \raw,
		\]
		and
		\[
		\rep \preceq \geom \preceq \raw .
		\]
		
		\item The full Borel modality $\Borfull$ is not canonically comparable with $\raw$ in general.
	\end{enumerate}
\end{theorem}

\begin{proof}
	\textup{(a)}
	A TTE transport has computably realized \(E\), \(D\), and uniformly computable \(\vartheta_\lambda\).  By the standing assumption on the represented spaces, computably
	realized maps are continuous. Hence every TTE transport is a continuous transport, i.e.
	\[
	\TTE\preceq\cont .
	\]
	Every continuous map between standard topological measurable spaces is Borel. Therefore a continuous encoding and continuous transcript reconstruction are also Borel, while the decoder remains continuous. Thus
	\[
	\cont\preceq\BorcD .
	\]
	Finally, \(\BorcD\) restricts \(E\) and \(\vartheta\) but still has continuous decoder \(D\). Since the raw modality allows all encodings and all transcript reconstructions with continuous decoder, every \(\BorcD\)-transport is raw, i.e.
	\[
	\BorcD\preceq\raw .
	\]
	
	\textup{(b)}
	A raw transport has continuous decoder. Every continuous map is Borel, and \(\BorD\) imposes no restriction on encodings or transcript reconstructions. Hence every
	raw transport is a decoder-only Borel transport, i.e.
	\[
	\raw\preceq\BorD .
	\]
	
	\textup{(c)}
	By definition,
	\[
	\mathcal R(\mathbf S, \mathbf P) \subseteq \mathcal G(\mathbf S, \mathbf P).
	\]
	Therefore every representation-preserving encoding is geometric. Also every geometric encoding admitted by the \(\mathfrak a\)-refinement is, in particular, an
	\(\mathfrak a\)-admissible encoding. The decoder and reconstruction components are the same as in \(\mathfrak a\). Hence
	\[
	\rep^{\mathfrak a} \preceq \geom^{\mathfrak a} \preceq \mathfrak a .
	\]
	The chains follow by substituting
	\[
	\mathfrak a=\TTE,\cont,\raw
	\]
	and using part \textup{(a)}.
	
	\textup{(d)}
	The full Borel modality is not automatically below \(\raw\), because it allows Borel decoders that are not continuous. Here is a concrete witness. Let
	\[
	h:\mathbb R\to\{0,1\}, \qquad h(x) :=
	\begin{cases}
		0,&x\le 0,\\
		1,&x>0.
	\end{cases}
	\]
	Let
	\[
	\mathcal S_h := (h,\mathbb R,\{0,1\},\{\operatorname{id}_{\mathbb R}\})
	\]
	and
	\[
	\mathcal P_{\mathbb R}:=(\operatorname{id}_{\mathbb R},\mathbb R,\mathbb R,
	\{\operatorname{id}_{\mathbb R}\}).
	\]
	Then
	\[
	\mathcal S_h \lemfq{\Borfull} \mathcal P_{\mathbb R}
	\]
	via
	\[
	E=\operatorname{id}_{\mathbb R}, \qquad D=h,
	\]
	since \(h\) is Borel. But
	\[
	\mathcal S_h \Nlemfq{\raw} \mathcal P_{\mathbb R}.
	\]
	Indeed, any raw transport would require a continuous decoder
	\[
	D:\mathbb R\to\{0,1\}
	\]
	such that
	\[
	D(E(x))=h(x).
	\]
	Because \(\mathbb R\) is connected and \(\{0,1\}\) is discrete, every continuous \(D:\mathbb R\to\{0,1\}\) is constant. Hence \(D(E(x))\) is constant, contradicting the nonconstancy of \(h\).
	
	Conversely, \(\raw\) is not automatically below \(\Borfull\), because \(\raw\) allows non-Borel encodings and non-Borel transcript reconstructions. Let
	\[
	A\subseteq[0,1]
	\]
	be non-Borel and nontrivial, and let
	\[
	\chi_A:[0,1]\to\{0,1\}
	\]
	be its characteristic function. Define
	\[
	\mathcal S_A := (\chi_A,[0,1],\{0,1\},\{\operatorname{id}_{[0,1]}\})
	\]
	and let
	\[
	\mathcal B := (\operatorname{id}_{\{0,1\}},\{0,1\},\{0,1\}, \{\operatorname{id}_{\{0,1\}}\})
	\]
	be the bit problem. Then
	\[
	\mathcal S_A \lemfq{\raw} \mathcal B
	\]
	via the set-theoretic encoding
	\[
	E(x)=\chi_A(x)
	\]
	and the identity decoder. If
	\[
	\mathcal S_A \lemfq{\Borfull} \mathcal B
	\]
	held, then there would be a Borel encoding
	\[
	E:[0,1]\to\{0,1\}
	\]
	and a Borel decoder
	\[
	D:\{0,1\}\to\{0,1\}
	\]
	with
	\[
	D(E(x))=\chi_A(x).
	\]
	But then
	\[
	A=(D\circ E)^{-1}(\{1\})
	\]
	would be Borel, a contradiction. Thus there is no general modality order between \(\Borfull\) and \(\raw\).
\end{proof}

The non-comparability of the full Borel modality with the raw modality is a serious warning. Changing the decoder class changes the limit behavior of transported towers; therefore Borel decoder modalities should not automatically be interpreted as raw type-\(G\)-sound.

\begin{remark}[TTE and bare geometry are generally orthogonal]\label{rem:TTE-geom-orthogonal}
	There is no canonical global comparison between the bare TTE modality and the bare geometric modality. TTE constrains implementability by computable realizers, while a geometric modality constrains the mathematical form of the encoding. A geometric transport may use fixed noncomputable geometric data, such as a noncomputable unitary or stabilizing operator, and therefore need not be TTE. Conversely, a TTE transport may be computably realized but geometrically artificial, for example by computably manufacturing an infinite object whose coordinates encode a predicate rather than by applying one of the allowed geometric operations.
	The comparable refinement is $\geom^{\TTE}$, which requires both geometric form and TTE-computable realization. Hence
	\[
	\geom^{\TTE}\preceq\TTE \qquad\text{and}\qquad \geom^{\TTE}\preceq\geom.
	\]
\end{remark}

\begin{remark}[Recommended modality diagram]\label{rem:recommended-modality-diagram}
	The useful order is a partial order, not a single linear chain. The most robust chains are
	\[
	\RepTTE \preceq \GeomTTE \preceq \TTE \preceq \cont \preceq \BorcD \preceq \raw,
	\]
	\[
	\RepCont \preceq \GeomCont \preceq \cont \preceq \raw,
	\]
	and
	\[
	\rep \preceq \geom \preceq \raw .
	\]
	The decoder-only Borel modality satisfies
	\[
	\raw\preceq\BorD,
	\]
	but \(\BorD\) is not automatically sound for raw type-\(G\) SCI lower-bound transfer, because its decoder need not commute with limits. The Borel modality most compatible
	with ordinary type-\(G\) limit pullbacks is $\BorcD$, where the decoder remains continuous.
\end{remark}

\begin{remark}[Computational strength of TTE reductions]\label{rem:TTE-strength}
	A TTE finite-query transport is stronger as a theorem but stricter as a preorder. More precisely,
	\[
	\mathbf S\lemfq{\TTE} \mathbf P \quad\Longrightarrow\quad
	\mathbf S\lemfq{\cont} \mathbf P \quad\Longrightarrow\quad \mathbf S\lemfq{\raw} \mathbf P,
	\]
	but the converses generally fail. Thus proving a TTE transport gives more information than proving a raw transport: it says that the finite-query simulation can be implemented
	uniformly on names. The price is that fewer transports exist.
\end{remark}

We now return to the CH23 diagonal sources. The next section shows that even the continuous and TTE finite-query modalities still identify them. Hence the missing structure is not mere topological or computable regularity, but geometric admissibility.

\section{Regularity Alone Does Not Recover The CH23 Distinction}\label{sec:regularity-collapse}
The raw collapse might seem to depend on the unrestricted nature of the raw encoding. This section shows that the same collapse persists under natural continuous and TTE
finite-query interpretations of the diagonal evaluation interface.

For the TTE statement, we must specify the represented spaces. We use the probably most direct choice: names are effective tables of the diagonal coefficients and singleton-window approximants.

\begin{definition}[Evaluation-name representation for the diagonal singleton problems]\label{def:evaluation-name-representation-diag}
	Assume that the diagonal singleton-window interface
	\[
	\Lambda_D = \{\mu_j:j\in\mathbb N\}\cup\{\rho_n:n\in\mathbb N\}
	\]
	is effectively indexed, where
	\[
	\mu_j(A,K)=a_j \quad\text{if}\quad Ae_j=a_j e_j,
	\]
	and
	\[
	\rho_n(A,K)=r_n(K)
	\]
	is the \(n\)-th rational approximant to the singleton compact input
	\[
	K=\{z\}\subset J.
	\]
	The evaluation-name representation of the diagonal singleton instance space is the represented-space structure in which a name of \((A,K)\) uniformly computes names of all
	evaluation values
	\[
	\mu_j(A,K),\qquad \rho_n(A,K),
	\]
	where $j,n\in\mathbb N$. Equivalently, a name is an oracle for the effective evaluation table of the instance.
	
	In this subsection, the TTE statements are made with respect to these evaluation-name representations and the standard representations of $\mathbb C,\, \mathbb R,\, \mathbb N,\, \{0,1\}$. We also assume that $\varepsilon>0$	and the fixed point $x_0\in J$ used in the collapse construction are computable, and that the fixed approximant sequence for
	\[
	L:=\{x_0\}
	\]
	is computable uniformly in \(n\).
\end{definition}

The following criterion converts uniform finite-transcript formulas for target evaluations into continuity and TTE-computability of the whole encoded instance.

\begin{lemma}[Effective finite-transcript criterion for evaluation-name encodings]\label{lem:effective-finite-transcript-criterion}
	Let \(\mathbf S\) and \(\mathbf P\) be effectively indexed typed SCI computational problems equipped with evaluation-name representations. Suppose that an encoding
	\[
	E:\Omega_S\to\Omega_P
	\]
	has the following property: uniformly in an index for a target evaluation
	\[
	\lambda\in\Lambda_P,
	\]
	one can compute a finite list of source evaluations
	\[
	\gamma_{\lambda,1},\ldots,\gamma_{\lambda,m_\lambda}\in\Lambda_S
	\]
	and a computable partial map
	\[
	\widehat\vartheta_\lambda: \subseteq V_{\gamma_{\lambda,1}}\times\cdots\times V_{\gamma_{\lambda,m_\lambda}} \to V_\lambda
	\]
	whose domain contains
	\[
	\operatorname{im}(\gamma_{\lambda,1},\ldots,\gamma_{\lambda,m_\lambda})
	\]
	and such that
	\[
	\lambda(E(x)) = \widehat\vartheta_\lambda \bigl( \gamma_{\lambda,1}(x),\ldots,\gamma_{\lambda,m_\lambda}(x) \bigr)
	\]
	for every \(x\in\Omega_S\). Then \(E\) has a computable realizer between the evaluation-name representations.
	
	Moreover, if all reconstruction maps above are continuous, then \(E\) is continuous for the initial evaluation topologies.
\end{lemma}

\begin{proof}
	We first prove computability. Let \(p\) be an evaluation-name of \(x\in\Omega_S\). To compute an evaluation-name of \(E(x)\), a TTE machine receives an index for a target
	evaluation \(\lambda\in\Lambda_P\). By hypothesis, uniformly from this index it computes indices for
	\[
	\gamma_{\lambda,1},\ldots,\gamma_{\lambda,m_\lambda}
	\]
	and an index for a computable partial realizer of $\widehat\vartheta_\lambda$.Using the source name \(p\), it computes names of the source evaluation values
	\[
	\gamma_{\lambda,1}(x),\ldots,\gamma_{\lambda,m_\lambda}(x).
	\]
	It then applies the computable realizer for \(\widehat\vartheta_\lambda\). The result is a name of $\lambda(E(x))$.	This procedure is uniform in \(\lambda\), hence computes an evaluation-name of \(E(x)\). Thus \(E\) has a computable realizer.
	
	For continuity, recall that the target instance space carries the initial topology with respect to the evaluations in \(\Lambda_P\). Hence \(E\) is continuous if and only if
	\[
	\lambda\circ E:\Omega_S\to V_\lambda
	\]
	is continuous for every \(\lambda\in\Lambda_P\). But
	\[
	\lambda(E(x)) = \widehat\vartheta_\lambda \bigl( \gamma_{\lambda,1}(x),\ldots,\gamma_{\lambda,m_\lambda}(x) \bigr),
	\]
	where each \(\gamma_{\lambda,i}\) is continuous by definition of the source initial topology and \(\widehat\vartheta_\lambda\) is continuous by assumption. Therefore \(\lambda\circ E\) is continuous for every \(\lambda\), and hence \(E\) is continuous.
\end{proof}

This criterion is exactly what the soft collapse encodings satisfy: each target diagonal coefficient is a continuous, computable function of one source diagonal coefficient and
one compact-window approximant.

For the moment we forget effectivity and retain only the initial evaluation topologies.

\begin{theorem}[Continuous collapse of the diagonal singleton sources]\label{thm:continuous-collapse-diagonal-sources}
	With respect to the initial evaluation topologies on the diagonal singleton-window instance spaces,
	\[
	\mathcal P^\sigma_{J,\diag} \equiv_{\cont,\mathrm{fq}} \mathcal P^{\sigma_\varepsilon}_{J,\varepsilon,\diag}.
	\]
\end{theorem}

\begin{proof}
	We use the two encodings from \cref{thm:raw-collapse}. We verify that they are continuous and that their finite transcript reconstructions are continuous.
	
	First consider the encoding
	\[
	E_{\sigma\to\varepsilon}(A,K)=(B(A,K),L), \qquad L=\{x_0\},
	\]
	where, writing
	\[
	K=\{z\}, \qquad Ae_j=a_j e_j,
	\]
	the diagonal entries of \(B(A,K)\) are indexed by pairs
	\[
	(p,j)\in\mathbb N \times\mathbb N_0
	\]
	and given by
	\[
	b_{p,j} = x_0+\varepsilon+|a_j-r_{p+2}(K)|+\frac1p .
	\]
	Fix once and for all a bijection
	\[
	\nu:\mathbb N\to\mathbb N \times\mathbb N_0, \qquad r\mapsto(p(r),j(r)).
	\]
	Thus \(B(A,K)\) is regarded as a diagonal operator on \(\ell^2(\mathbb N)\) after this fixed relabeling.
	
	Let \(\lambda\) be a target evaluation for the pseudospectral diagonal problem. If \(\lambda\) is a diagonal coefficient query at coordinate \(r\), then
	\[
	\nu(r)=(p,j)
	\]
	and
	\[
	\lambda(E_{\sigma\to\varepsilon}(A,K)) = b_{p,j} = x_0+\varepsilon+|\mu_j(A,K)-\rho_{p+2}(A,K)|+\frac1p .
	\]
	This is the composition of the finite source transcript
	\[
	(\mu_j,\rho_{p+2})
	\]
	with the continuous map
	\[
	(u,v)\longmapsto x_0+\varepsilon+|u-v|+\frac1p .
	\]
	If \(\lambda\) is an off-diagonal matrix-entry query, then
	\[
	\lambda(E_{\sigma\to\varepsilon}(A,K))=0,
	\]
	which is a constant continuous reconstruction. If \(\lambda\) is a compact-window approximant query, then
	\[
	\lambda(E_{\sigma\to\varepsilon}(A,K))
	\]
	is the fixed approximant to \(L=\{x_0\}\), again a constant continuous reconstruction. Thus every target evaluation in the compressed diagonal interface after \(E_{\sigma\to\varepsilon}\) is a continuous finite-transcript function of source evaluations.

	By \cref{lem:effective-finite-transcript-criterion}, the encoding \(E_{\sigma\to\varepsilon}\) is continuous. The decoder is the identity map on \(\{0,1\}\), hence continuous. The output identity was proved in \cref{thm:raw-collapse}. Therefore
	\[
	\mathcal P^\sigma_{J,\diag} \le_{\cont,\mathrm{fq}} \mathcal P^{\sigma_\varepsilon}_{J,\varepsilon,\diag}.
	\]
	
	For the reverse direction, use the encoding
	\[
	E_{\varepsilon\to\sigma}(A,K)=(C(A,K),L),
	\]
	where
	\[
	C(A,K)e_{p,j}=c_{p,j}e_{p,j},
	\]
	and
	\[
	c_{p,j} = x_0+\max\{0,|a_j-r_{p+2}(K)|-\varepsilon\}+\frac1p .
	\]
	A target diagonal coefficient at coordinate \(r\), with \(\nu(r)=(p,j)\), is
	\[
	c_{p,j} = x_0+ \max\{0,|\mu_j(A,K)-\rho_{p+2}(A,K)|-\varepsilon\} +\frac1p ,
	\]
	which is obtained from the finite transcript
	\[
	(\mu_j,\rho_{p+2})
	\]
	by the continuous map
	\[
	(u,v)\longmapsto x_0+\max\{0,|u-v|-\varepsilon\}+\frac1p .
	\]
	The compact-window approximant queries are constant, as above. Thus every target evaluation in the compressed diagonal interface after \(E_{\varepsilon\to\sigma}\) is a
	continuous finite-transcript function of source evaluations. Therefore \(E_{\varepsilon\to\sigma}\) is continuous. The decoder is again the identity, and the
	output identity was proved in \cref{thm:raw-collapse}. Hence
	\[
	\mathcal P^{\sigma_\varepsilon}_{J,\varepsilon,\diag} \le_{\cont,\mathrm{fq}} \mathcal P^\sigma_{J,\diag}.
	\]
	Combining the two reductions proves the continuous equivalence.
\end{proof}

With effective evaluation names, the same formulas are computable uniformly in the target coordinate.

\begin{theorem}[TTE finite-query collapse of the diagonal singleton sources]\label{thm:TTE-collapse-diagonal-sources}
	Assume the effective evaluation-name representations from \cref{def:evaluation-name-representation-diag}. Then
	\[
	\mathcal P^\sigma_{J,\diag} \equiv_{\TTE,\mathrm{fq}} \mathcal P^{\sigma_\varepsilon}_{J,\varepsilon,\diag}.
	\]
	Consequently, the same two problems are also equivalent in the continuous and raw modalities.
\end{theorem}

\begin{proof}
	We prove that the two collapse transports from \cref{thm:raw-collapse} satisfy the TTE finite-query requirements.
	
	For the first encoding
	\[
	E_{\sigma\to\varepsilon}(A,K)=(B(A,K),L),
	\]
	a target diagonal coefficient query at coordinate \(r\), with
	\[
	\nu(r)=(p,j),
	\]
	is computed from the source evaluations
	\[
	\mu_j(A,K)=a_j, \qquad \rho_{p+2}(A,K)=r_{p+2}(K)
	\]
	by
	\[
	(u,v) \longmapsto x_0+\varepsilon+|u-v|+\frac1p .
	\]
	This map is computable on standard representations of complex numbers and reals, since addition, subtraction, complex modulus, and addition of the computable constants $x_0,\, \varepsilon, \, 1/p$ are computable operations. The finite source transcript
	\[
	(\mu_j,\rho_{p+2})
	\]
	is computable uniformly from \(r\), because \(\nu\) is a fixed computable bijection.
	
	The target compact-window approximants for \(L=\{x_0\}\) are computable uniformly in the query index by the assumption on the fixed approximant sequence for \(L\).
	
	Thus the hypotheses of \cref{lem:effective-finite-transcript-criterion} hold effectively and uniformly for \(E_{\sigma\to\varepsilon}\). Hence \(E_{\sigma\to\varepsilon}\) has a
	computable realizer. The decoder is the identity on the discrete represented space \(\{0,1\}\), hence computable. The output identity is the one already proved in \cref{thm:raw-collapse}. Therefore
	\[
	\mathcal P^\sigma_{J,\diag} \le_{\TTE,\mathrm{fq}} \mathcal P^{\sigma_\varepsilon}_{J,\varepsilon,\diag}.
	\]
	
	For the reverse encoding
	\[
	E_{\varepsilon\to\sigma}(A,K)=(C(A,K),L),
	\]
	a target diagonal coefficient query at coordinate \(r\), with
	\[
	\nu(r)=(p,j),
	\]
	is computed from the same finite transcript
	\[
	(\mu_j,\rho_{p+2})
	\]
	by
	\[
	(u,v) \longmapsto x_0+\max\{0,|u-v|-\varepsilon\}+\frac1p .
	\]
	This map is computable because
	\[
	(s,t)\mapsto \max\{s,t\}
	\]
	is computable on represented real numbers, and the remaining operations are computable. Again the transcript is computable uniformly from \(r\), and all compact-input target queries are computable constant reconstructions.
	
	Hence \(E_{\varepsilon\to\sigma}\) has a computable realizer, the decoder is computable, and the output identity is the one proved in \cref{thm:raw-collapse}. Therefore
	\[
	\mathcal P^{\sigma_\varepsilon}_{J,\varepsilon,\diag} \le_{\TTE,\mathrm{fq}} \mathcal P^\sigma_{J,\diag}.
	\]
	The two TTE reductions give the claimed TTE equivalence. Since
	\[
	\TTE\preceq\cont\preceq\raw,
	\]
	the corresponding continuous and raw equivalences also follow, agreeing with \cref{thm:continuous-collapse-diagonal-sources} and \cref{thm:raw-collapse}.
\end{proof}

By \cref{thm:tte-fq-weihrauch-trace}, this yields a strong Weihrauch reduction between the represented target maps. The conclusion of \cref{thm:TTE-collapse-diagonal-sources} is stronger than Weihrauch reducibility, however, because the collapse encodings also provide uniform finite traces for all target evaluations.

\begin{remark}[Transport-side uniformity and the pure \(\mathcal R-\lim\) normal form]\label{rem:transport-side-uniformity}
	The uniformity clause in the TTE finite-query modality is the transport-side analogue of the pure \(\mathcal R-\lim\) uniformity requirement in the implemented SCI hierarchy.
	
	On the tower side, \cite[Rem.~2.20]{Sorg26Foundations} explains that even if all deepest-level approximants are individually computable, this does not give a represented-space computability model unless there is one uniform procedure producing the indexed approximation table. In the Weihrauch comparison, this is encoded by the pure \(\mathcal R-\lim\) normal form of \cite[Def.~A.21]{Sorg26Foundations}, and \cite[Thm.~A.22]{Sorg26Foundations} proves that the corresponding pure \(\mathcal R\)-tower height agrees with the \(\mathcal R\)-Weihrauch rank under the lim-normal-form hypothesis. The failure of deepest-only computability to imply such a pure witness is made explicit in \cite[Prop.~E.3]{Sorg26Foundations}.
	
	On the transport side, \cref{thm:tte-fq-weihrauch-trace} isolates the analogous condition: a TTE finite-query transport is not merely a family of pointwise finite simulations, but a uniform name-level implementation of the whole interface trace.
\end{remark}
	
	The collapse is therefore not merely an artifact of discontinuous transcript maps. The soft accumulation encodings use continuous, indeed computable, coordinate operations.
	What fails is not regularity in the topological or TTE sense; what fails is geometric naturality. The encodings manufacture a new infinite diagonal operator whose spectral
	accumulation pattern stores the source predicate. This explains why the geometric modality separates the two sources even though the raw, continuous, and TTE modalities
	identify them.

We now develop the modal exact-degree machinery needed to state precisely what geometric naturality recovers.

\section{Modal Exact-Degree Machinery}\label{sec:modal-degree-machinery}
Having introduced modalities, we can repeat the exact-basis construction at the modal level. This section is purely order-theoretic, except for the raw-soundness statements connecting modal witnesses back to raw type-\(G\) exactness.

\subsection{Modal Exact Degrees And Source-Separated Theorem Blocks}
The exact layer is still defined by raw type-\(G\) height. The modality changes only the transport degree relation inside that layer.

\begin{definition}[Exact layer and modal exact degrees]\label{def:modal-degrees}
	Let \(\mathcal U\) be a family of (typed) SCI computational problems, let \(k\in\bN\), and let \(\mathfrak m\) be an admissible modality. The raw exact layer is
	\[
	\cO_k(\mathcal U) := \{\mathbf P\in \mathcal U : \SCIG(\mathbf P)=k\}.
	\]
	For \(\mathbf P, \mathbf Q\in\cO_k(\mathcal U)\), write
	\[
	\mathbf P\equiv_{\mathfrak m,\fq} \mathbf Q
	\]
	if
	\[
	\mathbf P\lemfq{\mathfrak m} \mathbf Q \quad\text{and}\quad \mathbf Q\lemfq{\mathfrak m} \mathbf P.
	\]
	The set of modal exact degrees is
	\[
	D^{\mathfrak m}_k(\mathcal U) := \cO_k(\mathcal U)/{\equiv_{\mathfrak m,\fq}}.
	\]
	We write \([\mathbf P]_{\mathfrak m}\) for the modal exact degree of \(\mathbf P\), and define
	\[
	[\mathbf P]_{\mathfrak m} \preceq [\mathbf Q]_{\mathfrak m} \quad\Longleftrightarrow\quad \mathbf P\lemfq{\mathfrak m} \mathbf Q.
	\]
\end{definition}

\begin{definition}[Minimal modal exact degrees]\label{def:min-modal-degrees}
	A degree \(m\in D^{\mathfrak m}_k(\mathcal U)\) is minimal if there is no \(d\in D^{\mathfrak m}_k(\mathcal U)\) with \(d\prec m\). The set of minimal modal exact degrees is
	\[
	\operatorname{MinDeg}^{\mathfrak m}_k(\mathcal U).
	\]
\end{definition}

As usual, the preorder becomes a partial order after quotienting by mutual reducibility.

\begin{lemma}[Well-definedness of modal exact degrees]\label{lem:modal-welldefined}
	The relation
	\[
	[\mathbf P]_{\mathfrak m} \preceq [\mathbf Q]_{\mathfrak m} \quad:\Longleftrightarrow\quad
	\mathbf P\le_{\mathfrak m,\mathrm{fq}} \mathbf Q
	\]
	on
	\[
	D_k^{\mathfrak m}(\mathcal U) = \mathcal O_k(\mathcal U)/{\equiv_{\mathfrak m,\mathrm{fq}}}
	\]
	is well defined and is a partial order.
\end{lemma}

\begin{proof}
	We first prove well-definedness. Suppose
	\[
	\mathbf P \equiv_{\mathfrak m,\mathrm{fq}} \mathbf P', \qquad
	\mathbf Q\equiv_{\mathfrak m,\mathrm{fq}} \mathbf Q',
	\]
	and suppose
	\[
	\mathbf P\le_{\mathfrak m,\mathrm{fq}} \mathbf Q.
	\]
	Then
	\[
	\mathbf P'\le_{\mathfrak m,\mathrm{fq}} \mathbf P \le_{\mathfrak m,\mathrm{fq}} \mathbf Q \le_{\mathfrak m,\mathrm{fq}} \mathbf Q'
	\]
	by transitivity of the preorder. Hence
	\[
	\mathbf P'\le_{\mathfrak m,\mathrm{fq}} \mathbf Q'.
	\]
	Thus the order relation does not depend on the chosen representatives.
	
	Reflexivity follows because
	\[
	\mathbf P\le_{\mathfrak m,\mathrm{fq}} \mathbf P
	\]
	for every \(\mathbf P\). Transitivity follows directly from transitivity of $\le_{\mathfrak m,\mathrm{fq}}$. For antisymmetry, suppose
	\[
	[\mathbf P]_{\mathfrak m}\preceq [\mathbf Q]_{\mathfrak m} \qquad\text{and}\qquad
	[\mathbf Q]_{\mathfrak m}\preceq [\mathbf P]_{\mathfrak m}.
	\]
	Then
	\[
	\mathbf P \le_{\mathfrak m,\mathrm{fq}} \mathbf Q \qquad\text{and}\qquad \mathbf Q\le_{\mathfrak m,\mathrm{fq}} \mathbf P,
	\]
	so
	\[
	\mathbf P\equiv_{\mathfrak m,\mathrm{fq}} \mathbf Q.
	\]
	Therefore
	\[
	[\mathbf P]_{\mathfrak m}= [\mathbf Q]_{\mathfrak m}.
	\]
	Hence the quotient relation is a partial order.
\end{proof}

The finite theorem-block criterion packages the situation needed in the CH23 geometric proof: finitely many blocks, one source per block, within-block reductions, and cross-block non-reductions.

\begin{definition}[Modal source-separated finite theorem block]\label{def:modal-source-separated}
	Let \(\mathcal U\) be a finite family of (typed) SCI computational problems and 
	\[
	\mathcal U= \mathcal U_1\sqcup\cdots\sqcup \mathcal U_r
	\]
	be a partition, and choose sources
	\[
	\mathbf S_i \in \mathcal U_i,
	\]
	for $i=1,\ldots,r$.	This data is a \(\mathfrak m\)-source-separated finite theorem block at height \(k\) if
	\begin{enumerate}[label=\textup{(B\arabic*)}]
		\item every \(\mathbf P\in \mathcal U\) has \(\SCIG(\mathbf P)=k\);
		\item for every \(i\) and every \(\mathbf P\in \mathcal U_i\),
		\[
		\mathbf S_i \lemfq{\mathfrak m} \mathbf P;
		\]
		\item for \(i\ne j\) and every \(\mathbf P\in \mathcal U_j\),
		\[
		\mathbf S_i \nleq_{\mathfrak m,\fq} \mathbf P.
		\]
	\end{enumerate}
\end{definition}

The next theorem is the finite assembly step. Once the book of reductions and non-reductions is known, the minimal exact degrees follow formally.

\begin{theorem}[Finite source-separated theorem-block completion]\label{thm:finite-block}
	If \(\mathcal U= \mathcal U_1\sqcup\cdots\sqcup \mathcal U_r\) carries a \(\mathfrak m\)-source-separated finite theorem-block structure at height \(k\), then
	\[
	\operatorname{MinDeg}^{\mathfrak m}_k(\mathcal U) = \{[\mathbf S_1]_{\mathfrak m},\ldots,[\mathbf S_r]_{\mathfrak m}\}.
	\]
	Moreover, for each \(\mathbf P\in \mathcal U_i\), the unique minimal degree below \([\mathbf P]_{\mathfrak m}\) is \([\mathbf S_i]_{\mathfrak m}\).
\end{theorem}

\begin{proof}
	For each \(i\), condition \textup{(B2)} gives
	\[
	\mathbf S_i\le_{\mathfrak m,\mathrm{fq}} \mathbf P
	\]
	for $\mathbf P\in \mathcal U_i$. Hence
	\[
	[\mathbf S_i]_{\mathfrak m}\preceq [\mathbf P]_{\mathfrak m}.
	\]
	
	We first show that the source degrees are pairwise distinct. Suppose $i\neq j$.	If
	\[
	[\mathbf S_i]_{\mathfrak m}= [\mathbf S_j]_{\mathfrak m},
	\]
	then in particular
	\[
	\mathbf S_i \le_{\mathfrak m,\mathrm{fq}} \mathbf S_j.
	\]
	But \(\mathbf S_j\in \mathcal U_j\), so this contradicts condition \textup{(B3)} for the source \(\mathbf S_i\) and the block \(\mathcal U_j\). Thus
	\[
	[S_i]_{\mathfrak m}\neq [S_j]_{\mathfrak m} \quad \bigl( \forall i\neq j \bigr).
	\]
	
	Next let
	\[
	[\mathbf Q]_{\mathfrak m} \in D_k^{\mathfrak m}(\mathcal U).
	\]
	Then \(\mathbf Q\in \mathcal U_i\) for a unique \(i\). By \textup{(B2)},
	\[
	[\mathbf S_i]_{\mathfrak m} \preceq [\mathbf Q]_{\mathfrak m}.
	\]
	Thus every exact degree lies above one of the source degrees.
	
	We now prove that each \([\mathbf S_i]_{\mathfrak m}\) is minimal. Suppose
	\[
	[\mathbf Q]_{\mathfrak m} \preceq [\mathbf S_i]_{\mathfrak m}.
	\]
	Let \(\mathbf Q \in \mathcal U_j\). If \(i\neq j\), then \textup{(B2)} gives
	\[
	\mathbf S_j \le_{\mathfrak m,\mathrm{fq}} \mathbf Q.
	\]
	Together with
	\[
	\mathbf Q \le_{\mathfrak m,\mathrm{fq}} \mathbf S_i
	\]
	this implies
	\[
	\mathbf S_j \le_{\mathfrak m,\mathrm{fq}} \mathbf S_i,
	\]
	contradicting \textup{(B3)} for the source \(\mathbf S_j\) and the target block \(\mathcal U_i\), since \(\mathbf S_i \in \mathcal U_i\). Therefore \(j=i\). But then \textup{(B2)} gives
	\[
	\mathbf S_i \le_{\mathfrak m,\mathrm{fq}} \mathbf Q.
	\]
	Since also
	\[
	\mathbf Q \le_{\mathfrak m,\mathrm{fq}} \mathbf S_i,
	\]
	we have
	\[
	[\mathbf Q]_{\mathfrak m}= [\mathbf S_i]_{\mathfrak m}.
	\]
	So \([\mathbf S_i]_{\mathfrak m}\) is minimal.
	
	Finally, if \(\mathbf P\in \mathcal U_i\), then \([\mathbf S_i]_{\mathfrak m} \preceq [\mathbf P]_{\mathfrak m}\) by \textup{(B2)}. If another minimal source degree \([\mathbf S_j]_{\mathfrak m}\) lay below \([\mathbf P]_{\mathfrak m}\) with \(j\neq i\), then
	\[
	\mathbf S_j \le_{\mathfrak m,\mathrm{fq}} \mathbf P
	\]
	would contradict \textup{(B3)}. Hence the unique minimal degree below \([\mathbf P]_{\mathfrak m}\) is \([\mathbf S_i]_{\mathfrak m}\).
\end{proof}

The principal case is the degenerate one-block version. It is used for the raw six-problem collapse.

\begin{corollary}[Principal collapse criterion]\label{cor:principal-collapse}
	Let \(\mathcal U\) be finite, let every member of \(\mathcal U\) have exact raw height \(k\), and suppose that there is \(\mathbf S\in \mathcal U\) such that
	\[
	\mathbf S \lemfq{\mathfrak m} \mathbf P
	\]
	for $\mathbf P\in \mathcal U$. Then
	\[
	\operatorname{MinDeg}^{\mathfrak m}_k(\mathcal U)=\{[\mathbf S]_{\mathfrak m}\}.
	\]
\end{corollary}

\begin{proof}
	Since every member of \(\mathcal U\) has exact raw height \(k\), every member of \(\mathcal U\) represents an element of \(D_k^{\mathfrak m}(\mathcal U)\). By assumption,
	\[
	\mathbf S \le_{\mathfrak m,\mathrm{fq}} \mathbf P
	\]
	for $\mathbf P\in \mathcal U$. Hence
	\[
	[\mathbf S]_{\mathfrak m} \preceq [\mathbf P]_{\mathfrak m}
	\]
	for $\mathbf P\in \mathcal U$. Let \([\mathbf Q]_{\mathfrak m}\) be a minimal element of \(D_k^{\mathfrak m}(\mathcal U)\). Then
	\[
	[\mathbf S]_{\mathfrak m} \preceq [\mathbf Q]_{\mathfrak m}.
	\]
	By minimality of \([\mathbf Q]_{\mathfrak m}\), this forces
	\[
	[\mathbf S]_{\mathfrak m}= [\mathbf Q]_{\mathfrak m}.
	\]
	Therefore the only minimal degree is \([\mathbf S]_{\mathfrak m}\).
\end{proof}

The remaining part of this section explains when modal witnesses still prove raw family-pointwise exactness.

\subsection{Raw-Sound Modal Bases And Family-Pointwise Exactness}
Not every modality is sound for raw type-\(G\) height. The relevant modalities for raw exactness are those lying below the raw preorder.

\begin{definition}[Raw-sound transport modality]\label{def:raw-sound-modality}
	An admissible transport modality \(\mathfrak m\) is called raw-sound if
	\[
	\mathfrak m\preceq \raw.
	\]
	Equivalently, every \(\mathfrak m\)-finite-query transport is also a raw type-\(G\)
	finite-query transport.
\end{definition}

The following lemma is the basic soundness mechanism inherited from a finite-query pullback.

\begin{lemma}[Raw finite-query monotonicity of type-\(G\) SCI]\label{lem:raw-fq-monotonicity}
	Let
	\[
	\mathbf S=(\Psi,\Sigma,(\mathcal N,\rho),\Lambda_S), \qquad
	\mathbf P=(\Xi,\Omega,(\mathcal M,d),\Lambda_P)
	\]
	be SCI typed computational problems. If
	\[
	\mathbf S \le_{\raw,\mathrm{fq}} \mathbf P,
	\]
	then
	\[
	\mathrm{SCI}_G(\mathbf S) \le \mathrm{SCI}_G(\mathbf P).
	\]
\end{lemma}

\begin{proof}
	Let
	\[
	h:=\mathrm{SCI}_G(\mathbf P).
	\]
	If \(h=\infty\), there is nothing to prove. Assume \(h<\infty\). If \(h=0\), let
	\[
	\Gamma:\Omega\to \mathcal M
	\]
	be a general algorithm with
	\[
	\Gamma(B)=\Xi(B)
	\]
	for $B\in\Omega$. If \(h\ge1\), let
	\[
	\Gamma_{n_h,\ldots,n_1}:\Omega\to \mathcal M
	\]
	be a raw type-\(G\) tower of height \(h\) for \(\mathbf P\). Thus
	\[
	\Xi(B) = \lim_{n_h\to\infty}\cdots\lim_{n_1\to\infty} \Gamma_{n_h,\ldots,n_1}(B)
	\]
	for $B\in\Omega$. Now let
	\[
	E:\Sigma\to\Omega, \qquad
	D:(\mathcal M,d) \to (\mathcal N,\rho)
	\]
	and the finite reconstruction packages
	\[
	f(E(A)) = \vartheta_f \bigl( \gamma_{f,1}(A),\ldots,\gamma_{f,m_f}(A) \bigr)
	\]
	for $f\in\Lambda_P$, witness
	\[
	\mathbf S \le_{\raw,\mathrm{fq}} \mathbf P.
	\]
	We write \(\Gamma_{\mathbf n}\) for either \(\Gamma\) in the case \(h=0\), or for \(\Gamma_{n_h,\ldots,n_1}\) in the case \(h\ge1\). Define
	\[
	\widetilde\Gamma_{n_h,\ldots,n_1}(A) := D\bigl(\Gamma_{n_h,\ldots,n_1}(E(A))\bigr),
	\]
	where $A\in\Sigma$.
	
	We first check that each
	\[
	\widetilde\Gamma_{n_h,\ldots,n_1}
	\]
	is a raw general algorithm for \(\mathbf S\). Fix \(A\in\Sigma\). Let
	\[
	F_A := \Lambda_{\Gamma_{n_h,\ldots,n_1}}(E(A)) \subseteq \Lambda_P
	\]
	be the finite target query set used by the target general algorithm at \(E(A)\). Define the source query set
	\[
	\widetilde\Lambda(A) := \bigcup_{f\in F_A} \{\gamma_{f,1},\ldots,\gamma_{f,m_f}\}.
	\]
	This is a finite, possibly empty, subset of \(\Lambda_S\).
	
	Suppose \(A'\in\Sigma\) agrees with \(A\) on all evaluations in
	\[
	\widetilde\Lambda(A).
	\]
	Then, for every \(f\in F_A\),
	\[
	f(E(A')) = \vartheta_f \bigl( \gamma_{f,1}(A'),\ldots,\gamma_{f,m_f}(A') \bigr) = \vartheta_f \bigl( \gamma_{f,1}(A),\ldots,\gamma_{f,m_f}(A) \bigr) = f(E(A)).
	\]
	Since \(\Gamma_{n_h,\ldots,n_1}\) is a general algorithm, this implies
	\[
	\Gamma_{n_h,\ldots,n_1}(E(A')) = \Gamma_{n_h,\ldots,n_1}(E(A)),
	\]
	and therefore
	\[
	\widetilde\Gamma_{n_h,\ldots,n_1}(A') = \widetilde\Gamma_{n_h,\ldots,n_1}(A).
	\]
	It also implies that the target query set at \(E(A')\) equals the target query set at \(E(A)\), hence the source query set \(\widetilde\Lambda(A')\) equals \(\widetilde\Lambda(A)\).  Thus
	\[
	\widetilde\Gamma_{n_h,\ldots,n_1}
	\]
	satisfies the two defining finite-information conditions for a raw general algorithm.
	
	It remains to check convergence. If \(h=0\), then
	\[
	\widetilde\Gamma(A) := D(\Gamma(E(A))) = D(\Xi(E(A))) = \Psi(A),
	\]
	so \(\mathbf S\) has height \(0\).
	
	If \(h\ge1\), then for every \(A\in\Sigma\),
	\[
	\Psi(A) = D(\Xi(E(A))).
	\]
	Since \(D\) is continuous and
	\[
	\Xi(E(A)) = \lim_{n_h\to\infty}\cdots\lim_{n_1\to\infty} \Gamma_{n_h,\ldots,n_1}(E(A)),
	\]
	we may pass \(D\) through the iterated limits, one limit at a time, and obtain
	\[
	\Psi(A) = \lim_{n_h\to\infty}\cdots\lim_{n_1\to\infty} D(\Gamma_{n_h,\ldots,n_1}(E(A))).
	\]
	Thus \(\mathbf S\) has a raw type-\(G\) tower of height \(h\). Therefore
	\[
	\mathrm{SCI}_G(\mathbf S) \le h=\mathrm{SCI}_G(\mathbf P).
	\]
\end{proof}

\begin{lemma}[Raw-height monotonicity for raw-sound modalities]\label{lem:raw-sound-height-monotonicity}
	Let \(\mathfrak m\) be raw-sound. If
	\[
	\mathbf S \le_{\mathfrak m,\mathrm{fq}} \mathbf P,
	\]
	then
	\[
	\mathrm{SCI}_G(\mathbf S) \le \mathrm{SCI}_G(\mathbf P).
	\]
\end{lemma}

\begin{proof}
	Since \(\mathfrak m\preceq\raw\), the transport
	\[
	\mathbf S \le_{\mathfrak m,\mathrm{fq}} \mathbf P
	\]
	implies
	\[
	\mathbf S \le_{\raw,\mathrm{fq}} \mathbf P.
	\]
	The conclusion follows from \cref{lem:raw-fq-monotonicity}.
\end{proof}

We can now define modal exact bases in a way that still yields raw height information.

\begin{definition}[Modal exact basis]\label{def:modal-exact-basis}
	Let \(\mathcal U\) be a family of (typed) SCI computational problems, let \(k\in\mathbb N\), and let \(\mathfrak m\preceq\raw\) be raw-sound. Set
	\[
	\mathcal U_{\le k} := \{\mathbf P \in \mathcal U : \mathrm{SCI}_G(\mathbf P) \le k\}.
	\]
	A modal exact basis for \((\mathcal U,k,\mathfrak m)\) is a set
	\[
	B^{\mathfrak m}_k(\mathcal U)\subseteq\mathcal O_k(\mathcal U)
	\]
	such that
	\[
	\mathcal O_k(\mathcal U) = \{\mathbf P\in \mathcal U_{\le k}: \exists \mathbf B \in B^{\mathfrak m}_k(\mathcal U) \text{ with } \mathbf B \le_{\mathfrak m,\mathrm{fq}} \mathbf P \}.
	\]
\end{definition}

\begin{remark}[Why the truncation \(\mathcal U_{\le k}\) is necessary]\label{rem:truncation}
	Without the truncation \(\mathcal U_{\le k}\), an exact height-\(k\) source could reduce to a higher-height target. The exact-layer equality would then fail for a reason unrelated to exact support. The correct exact-basis problem is therefore a problem inside the \(k\)-bounded truncation of the ambient.
\end{remark}

The theorem below is the modal version of the witness-sharpness upgrade: basis coverage is not merely sufficient for family-pointwise exactness; inside a \(k\)-bounded ambient it is also necessary by definition of the basis.

\begin{theorem}[Modal witness-basis criterion for family-pointwise exactness]\label{thm:modal-witness-basis-family-exactness}
	Let
	\[
	\mathfrak m\preceq\raw
	\]
	be raw-sound and \(\mathcal U\) be a \(k\)-bounded family of (typed) SCI computational problems, i.e.
	\[
	\mathcal U= \mathcal U_{\le k}.
	\]
	Let
	\[
	B^{\mathfrak m}_k(\mathcal U)\subseteq\mathcal O_k(\mathcal U)
	\]
	be a modal exact basis for \((\mathcal U,k,\mathfrak m)\). Then for every family
	\[
	\mathcal F \subseteq \mathcal U,
	\]
	the following are equivalent.
	
	\begin{enumerate}[label=\textup{(\roman*)}]
		\item \(\mathcal F\) is family-pointwise exact at height \(k\), i.e.
		\[
		\mathcal F\subseteq\mathcal O_k(\mathcal U).
		\]
		
		\item \(\mathcal F\) is covered by the modal basis, i.e.
		\[
		\bigl( \forall \mathbf P \in \mathcal F \bigr) \ \bigl( \exists \mathbf B\in B^{\mathfrak m}_k(\mathcal U) \bigr) \ \mathbf B\le_{\mathfrak m,\mathrm{fq}} \mathbf P.
		\]
	\end{enumerate}
\end{theorem}

\begin{proof}
	Assume first that
	\[
	\mathcal F \subseteq \mathcal O_k(\mathcal U).
	\]
	Let \(\mathbf P\in \mathcal F\). Since \(\mathbf P \in\mathcal O_k(\mathcal U)\), and since $B^{\mathfrak m}_k(\mathcal U)$	is a modal exact basis,
	\[
	\mathbf P \in \bigl\{ \mathbf Q \in \mathcal U_{\le k} : \exists \mathbf B\in B^{\mathfrak m}_k(\mathcal U) \text{ with } \mathbf B\le_{\mathfrak m,\mathrm{fq}} \mathbf Q \bigr\}.
	\]
	Hence there exists
	\[
	\mathbf B \in B^{\mathfrak m}_k(\mathcal U)
	\]
	such that
	\[
	\mathbf B \le_{\mathfrak m,\mathrm{fq}} \mathbf P.
	\]
	Thus \(\mathcal F\) is covered by the modal basis.
	
	Conversely, assume that \(\mathcal F\) is covered by the modal basis. Let \(\mathbf P \in \mathcal F\). Choose
	\[
	\mathbf B \in B^{\mathfrak m}_k(\mathcal U)
	\]
	such that
	\[
	\mathbf B \le_{\mathfrak m,\mathrm{fq}} \mathbf P.
	\]
	Because
	\[
	B^{\mathfrak m}_k(\mathcal U) \subseteq \mathcal O_k(\mathcal U),
	\]
	we have
	\[
	\mathrm{SCI}_G(\mathbf B)=k.
	\]
	By raw-soundness and \cref{lem:raw-sound-height-monotonicity},
	\[
	k=\mathrm{SCI}_G(\mathbf B) \le \mathrm{SCI}_G(\mathbf P).
	\]
	But
	\[
	\mathbf P \in \mathcal F \subseteq \mathcal U= \mathcal U_{\le k},
	\]
	so
	\[
	\mathrm{SCI}_G(\mathbf P) \le k.
	\]
	Therefore
	\[
	\mathbf P \in\mathcal O_k(\mathcal U).
	\]
	Since \(\mathbf P\in \mathcal F\) was arbitrary,
	\[
	\mathcal F \subseteq \mathcal O_k(\mathcal U).
	\]
\end{proof}

\begin{corollary}[Sufficient modal witnesses for family-pointwise exactness]\label{cor:sufficient-modal-witnesses-family-exactness}
	Let $\mathfrak m\preceq\raw$ be raw-sound, let \(\mathcal U= \mathcal U_{\le k}\), and let $\mathcal F \subseteq \mathcal U$. Suppose there exists a set
	\[
	\mathcal W \subseteq \mathcal O_k(\mathcal U)
	\]
	such that for every \(\mathbf P\in \mathcal F\) there is \(\mathbf W_P \in \mathcal W\) with
	\[
	\mathbf W_P \le_{\mathfrak m,\mathrm{fq}} \mathbf P.
	\]
	Then
	\[
	\mathcal F \subseteq\mathcal O_k(\mathcal U).
	\]
\end{corollary}

\begin{proof}
	For each \(\mathbf P \in \mathcal F\), choose \(\mathbf W_P \in \mathcal W\) with
	\[
	\mathbf W_P \le_{\mathfrak m,\mathrm{fq}} \mathbf P.
	\]
	Since \(\mathbf W_P \in \mathcal O_k(\mathcal U)\), $\mathrm{SCI}_G(\mathbf W_P)=k$. By raw-soundness,
	\[
	k=\mathrm{SCI}_G(\mathbf W_P) \le \mathrm{SCI}_G(\mathbf P).
	\]
	Since \(\mathbf P \in \mathcal U= \mathcal U_{\le k}\), $\mathrm{SCI}_G(\mathbf P)\le k$. Hence
	\[
	\mathrm{SCI}_G(\mathbf P)=k.
	\]
	Thus every \(\mathbf P\in \mathcal F\) has exact height \(k\).
\end{proof}

This criterion is what makes modal exact bases useful also for potential applications: once a canonical set of exact sources is known, checking exactness of a family reduces to checking modal transports from those sources.

\begin{remark}[What changes for non-raw-sound modalities]\label{rem:non-raw-sound-modalities}
	The raw-sound hypothesis
	\[
	\mathfrak m\preceq\raw
	\]
	is essential for statements about raw type-\(G\) exact height. For modalities not below \(\raw\), such as decoder-only Borel modalities with non-continuous decoders, the same
	argument need not apply because the transport may not preserve raw type-\(G\) limits. For such modalities, the corresponding exact-basis problem should be formulated relative to an implementation height $\mathrm{SCI}_C$ for which the modality is sound:
	\[
	\mathbf S \le_{\mathfrak m,\mathrm{fq}} \mathbf P \quad\Longrightarrow\quad
	\mathrm{SCI}_C(\mathbf S) \le \mathrm{SCI}_C(\mathbf P).
	\]
	Thus the raw-sound case is the correct setting for modal refinements of raw family-pointwise exactness, while other modalities require their own implementation semantics.
\end{remark}

\section{Geometric Recovery Of The CH23 Theorem Block}\label{sec:geometric-recovery}
We now define the geometric modality for the CH23 theorem block and prove that it recovers the expected two-source structure. The point is not to forbid the raw collapse
by declaration, but to restrict encodings to the natural operator-theoretic operations used in the spectral representations.

\subsection{The CH23 Neutral Geometric Modality}
The allowed geometric encodings preserve the spectral window predicates. Neutral stabilizations are permitted only when the stabilizing operator is fixed and invisible to all singleton windows in \(J\), both spectrally and pseudospectrally.

\begin{definition}[CH23 representation labels and realized operators]\label{def:ch23-realized-operator-classes}
	Fix a connected countable graph
	\[
	\mathcal G=(V,E_{\mathcal G}),\qquad V=\{v_1,v_2,\ldots\}.
	\]
	Let
	\[
	\mathfrak R:=\{\mathrm{diag},\mathrm{gen},\mathrm{graph}\}
	\]
	be the three CH23 representation labels used in the singleton-window block. Put
	\[
	H_{\mathrm{diag}}:=\ell^2(\mathbb N),\qquad
	H_{\mathrm{gen}}:=\ell^2(\mathbb N),\qquad
	H_{\mathrm{graph}}:=\ell^2(V).
	\]
	Let
	\[
	\Omega_{\mathrm{diag}},\qquad
	\Omega_{\mathrm{gen}},\qquad
	\Omega_{\mathrm{graph}}
	\]
	denote respectively the diagonal, general \(\ell^2(\mathbb N)\), and graph CH23 operator instance classes. For each representation label
	\[
	\rho\in\mathfrak R
	\]
	we write
	\[
	\operatorname{op}_\rho:\Omega_\rho\longrightarrow \mathcal C(H_\rho)
	\]
	for the map which sends a represented CH23 instance to its realized closed densely defined operator on \(H_\rho\). Here \(\mathcal C(H_\rho)\) denotes the class of
	closed densely defined linear operators on \(H_\rho\). Thus
	\[
	\operatorname{op}_{\mathrm{diag}}(A)=A,\qquad
	\operatorname{op}_{\mathrm{gen}}(A)=A,
	\]
	while \(\operatorname{op}_{\mathrm{graph}}(A)\) is the closed operator on \(\ell^2(V)\) represented by the graph-coefficient data of \(A\).
	
	For a compact interval \(J\subseteq\mathbb R\), define the singleton instance spaces
	\[
	\mathcal X_\rho(J):=\Omega_\rho\times \mathcal K_{\mathrm{sgl}}(J), \qquad \mathcal K_{\mathrm{sgl}}(J):=\{\{z\}:z\in J\}.
	\]
	If \((A,K)\in\mathcal X_\rho(J)\), we shall often write
	\[
	\sigma(A),\qquad \sigma_\varepsilon(A)
	\]
	as shorthand for
	\[
	\sigma(\operatorname{op}_\rho(A)),\qquad
	\sigma_\varepsilon(\operatorname{op}_\rho(A)).
	\]
	This convention is only notational: all spectral and pseudospectral statements are about the realized operator \(\operatorname{op}_\rho(A)\).
\end{definition}

\begin{definition}[CH23 \((J,\varepsilon)\)-neutral geometric encodings]\label{def:geom-encoding}
	Fix a compact interval $J\subseteq\mathbb R$ with nonempty interior and fix $\varepsilon>0$. A CH23 \((J,\varepsilon)\)-neutral elementary encoding of type
	\[
	\rho\longrightarrow\rho', \qquad
	\rho,\rho'\in\mathfrak R,
	\]
	is a map
	\[
	e:\mathcal X_\rho(J)\longrightarrow \mathcal X_{\rho'}(J), \qquad e(A,K)=(A',K),
	\]
	of one of the following forms. In all cases, the data defining the map are fixed once and for all and are independent of the input \((A,K)\).
	
	\begin{enumerate}[label=\textup{(E\arabic*)}]
		\item \textbf{Unitary conjugacy:}
		There is a fixed unitary operator
		\[
		U:H_\rho\longrightarrow H_{\rho'}
		\]
		and a fixed map
		\[
		R_U:\Omega_\rho\longrightarrow \Omega_{\rho'}
		\]
		such that, for every \(A\in\Omega_\rho\),
		\[
		\operatorname{op}_{\rho'}(R_U(A)) = U\,\operatorname{op}_\rho(A)\,U^{-1}.
		\]
		The elementary encoding is
		\[
		e(A,K):=(R_U(A),K).
		\]
		This generator is admissible only for source-target types \(\rho\to\rho'\) for which the fixed representation map \(R_U\) is part of the specified CH23 geometric
		structure.
		
		\item \textbf{Basis relabeling on \(\ell^2(\mathbb N)\):}
		Here
		\[
		\rho,\rho'\in\{\mathrm{diag},\mathrm{gen}\}.
		\]
		Let
		\[
		\pi:\mathbb N\longrightarrow\mathbb N
		\]
		be a fixed bijection and let
		\[
		U_\pi:\ell^2(\mathbb N)\longrightarrow\ell^2(\mathbb N), \qquad U_\pi e_j=e_{\pi(j)}
		\]
		be the corresponding permutation unitary. A basis relabeling is the elementary encoding
		\[
		e(A,K):=(R_\pi(A),K),
		\]
		where
		\[
		R_\pi:\Omega_\rho\longrightarrow\Omega_{\rho'}
		\]
		is a fixed representation map satisfying
		\[
		\operatorname{op}_{\rho'}(R_\pi(A)) = U_\pi\,\operatorname{op}_\rho(A)\,U_\pi^{-1}
		\]
		for $A\in\Omega_\rho$. For example, in the diagonal-to-diagonal case, if
		\[
		\operatorname{op}_{\mathrm{diag}}(A)e_j=a_j e_j,
		\]
		then \(R_\pi(A)\) is the diagonal operator with diagonal coefficients $a_{\pi^{-1}(j)}$.
		
		\item \textbf{Graph relabeling:}
		Here
		\[
		\rho=\rho'=\mathrm{graph}.
		\]
		Let
		\[
		\alpha:V\longrightarrow V
		\]
		be a fixed graph automorphism and let
		\[
		U_\alpha:\ell^2(V)\longrightarrow\ell^2(V), \qquad U_\alpha \mathbf 1_v=\mathbf 1_{\alpha(v)}
		\]
		be the induced unitary. A graph relabeling is the elementary encoding
		\[
		e(A,K):=(R_\alpha(A),K),
		\]
		where
		\[
		R_\alpha:\Omega_{\mathrm{graph}}\longrightarrow \Omega_{\mathrm{graph}}
		\]
		is the fixed relabeled graph representation satisfying
		\[
		\operatorname{op}_{\mathrm{graph}}(R_\alpha(A)) = U_\alpha\,\operatorname{op}_{\mathrm{graph}}(A)\,U_\alpha^{-1},
		\]
		where $A\in\Omega_{\mathrm{graph}}$.
		
		\item \textbf{Representation inclusion:}
		There are fixed data
		\[
		W_{\rho,\rho'}:H_\rho\longrightarrow H_{\rho'}
		\]
		unitary, and
		\[
		\iota_{\rho,\rho'}:\Omega_\rho\longrightarrow \Omega_{\rho'}
		\]
		such that, for every \(A\in\Omega_\rho\),
		\[
		\operatorname{op}_{\rho'}(\iota_{\rho,\rho'}(A)) = W_{\rho,\rho'}\,\operatorname{op}_\rho(A)\,W_{\rho,\rho'}^{-1}.
		\]
		The elementary encoding is
		\[
		e(A,K):=(\iota_{\rho,\rho'}(A),K).
		\]
		The representation inclusions we use in this work are
		\[
		\iota_{\mathrm{diag},\mathrm{gen}}:\Omega_{\mathrm{diag}}\to\Omega_{\mathrm{gen}}, \qquad
		W_{\mathrm{diag},\mathrm{gen}}=\operatorname{id}_{\ell^2(\mathbb N)},
		\]
		which views a diagonal operator as a member of the general CH23 class, and
		\[
		\iota_{\mathrm{diag},\mathrm{graph}}:\Omega_{\mathrm{diag}}\to\Omega_{\mathrm{graph}}, \qquad
		W_{\mathrm{diag},\mathrm{graph}}e_j=\mathbf 1_{v_j},
		\]
		which represents the same diagonal operator on the graph Hilbert space \(\ell^2(V)\). Explicitly, if
		\[
		Ae_j=a_j e_j,
		\]
		then the graph representation \(\iota_{\mathrm{diag},\mathrm{graph}}(A)\) has graph coefficients
		\[
		\alpha(v_i,v_j)=
		\begin{cases}
			a_j, & i=j,\\
			0, & i\neq j,
		\end{cases}
		\]
		with local support sets
		\[
		S_{v_j}=\{v_j\}.
		\]
		
		\item \textbf{Neutral direct-sum stabilization:}
		There are fixed data consisting of a Hilbert space \(H_0\), a closed densely defined operator
		\[
		T_0\in\mathcal C(H_0),
		\]
		a unitary
		\[
		W:H_\rho\oplus H_0\longrightarrow H_{\rho'},
		\]
		and a fixed representation map
		\[
		R_{T_0,W}:\Omega_\rho\longrightarrow \Omega_{\rho'}
		\]
		such that, for every \(A\in\Omega_\rho\),
		\[
		\operatorname{op}_{\rho'}(R_{T_0,W}(A)) = W\bigl(\operatorname{op}_\rho(A)\oplus T_0\bigr)W^{-1}.
		\]
		The elementary encoding is
		\[
		e(A,K):=(R_{T_0,W}(A),K).
		\]
		The stabilizing operator \(T_0\) must be \((J,\varepsilon)\)-neutral, meaning
		\[
		\sigma(T_0)\cap J=\varnothing, \qquad \sigma_\varepsilon(T_0)\cap J=\varnothing.
		\]
		Equivalently, for every singleton \(K=\{z\}\subset J\),
		\[
		\sigma(T_0)\cap K=\varnothing, \qquad \sigma_\varepsilon(T_0)\cap K=\varnothing.
		\]
	\end{enumerate}
	
	A CH23 \((J,\varepsilon)\)-neutral geometric encoding of type
	\[
	\rho\longrightarrow\rho'
	\]
	is any finite composition
	\[
	E=e_N\circ\cdots\circ e_1
	\]
	of CH23 \((J,\varepsilon)\)-neutral elementary encodings with compatible intermediate types
	\[
	\rho=\rho_0\longrightarrow\rho_1\longrightarrow\cdots\longrightarrow\rho_N=\rho'.
	\]
	For \(N=0\), the identity map on \(\mathcal X_\rho(J)\) is allowed. Since each elementary encoding leaves the compact input unchanged, every generated encoding has the form
	\[
	E(A,K)=(A_E,K).
	\]
	
	No representation selector is admissible merely because some abstract represented copy of an operator exists. The maps
	\[
	R_U,\quad R_\pi,\quad R_\alpha,\quad \iota_{\rho,\rho'},\quad R_{T_0,W}
	\]
	are part of the fixed CH23 geometric structure and are not allowed to depend on the particular input \((A,K)\) except through the operator-theoretic formulae.
\end{definition}

\begin{definition}[The CH23 singleton-window geometric modality]\label{def:geomCH23-modality}
	Let
	\[
	\mathcal U^\sharp_{2,J,\varepsilon} =
	\{\mathcal P^\sigma_{J,\mathrm{diag}}, \mathcal P^\sigma_{J,\mathrm{gen}}, \mathcal P^\sigma_{J,\mathrm{graph}},
	\mathcal P^{\sigma\varepsilon}_{J,\varepsilon,\mathrm{diag}},
	\mathcal P^{\sigma\varepsilon}_{J,\varepsilon,\mathrm{gen}},
	\mathcal P^{\sigma\varepsilon}_{J,\varepsilon,\mathrm{graph}}\}.
	\]
	For
	\[
	\mathbf S,\mathbf P\in \mathcal U^\sharp_{2,J,\varepsilon},
	\]
	let
	\[
	\rho(\mathbf S),\rho(\mathbf P)\in\mathfrak R
	\]
	denote their representation labels. The encoding class of the CH23 geometric modality is
	\[
	\mathcal E_{\mathrm{geom}^{\mathrm{CH23}}_{J,\varepsilon}}(\mathbf S,\mathbf P) :=
	\{E:\mathcal X_{\rho(\mathbf S)}(J)\to \mathcal X_{\rho(\mathbf P)}(J): E\text{ is a CH23 }(J,\varepsilon)\text{-neutral geometric encoding}\}.
	\]
	The decoder class is
	\[
	\mathcal D_{\mathrm{geom}^{\mathrm{CH23}}_{J,\varepsilon}}(\mathbf P,\mathbf S) := C(Y_{P},Y_{S}),
	\]
	and the reconstruction class is the raw finite-transcript class:
	\[
	\Theta_{\mathrm{geom}^{\mathrm{CH23}}_{J,\varepsilon}}(\lambda;\vec\gamma) :=
	\{\vartheta:\operatorname{im}(\vec\gamma)\to V_\lambda\}.
	\]
	Thus the CH23 geometric modality restricts only the instance encoding; decoders remain continuous and transcript reconstructions remain arbitrary finite maps.
	
	If a global modality on all typed SCI problems is desired, extend this local modality by allowing only identity encodings outside the full CH23 singleton-window subcategory.
	All arguments in \cref{sec:geometric-recovery} use only the local part on \(\mathcal U^\sharp_{2,J,\varepsilon}\).
\end{definition}

\Cref{def:geom-encoding} and \cref{def:geomCH23-modality} permit the diagonal-to-general and diagonal-to-graph representation moves used below, but they
exclude the artificial encodings from \cref{thm:raw-collapse}, whose purpose was to store a predicate in a newly manufactured accumulation pattern. Those raw encodings are not unitary conjugacies, relabelings, representation inclusions, or neutral stabilizations of the original operator.

\begin{remark}[Why this is not circular]\label{rem:not-circular-geom}
	The modality is not defined as ``the maps for which the theorem is true.'' It is generated by standard operator-theoretic operations: unitary equivalence, representation inclusion, relabeling, and fixed neutral stabilization. The forbidden move is precisely the raw coding move from \cref{thm:raw-collapse}: constructing a new infinite diagonal operator whose accumulation pattern depends on the source predicate.
\end{remark}

The key structural property of this modality is predicate preservation.

The spectral and pseudospectral identities used below should be understood in the closed-operator sense. For matrices, unitary invariance of the resolvent norm and pseudospectra is recorded in \cite[§2, p.~17, (2.11)-(2.12)]{trefethen2020spectra}, and the matrix direct-sum pseudospectral identity appears in
\cite[§2, p.~20, Thm.~2.4(iii)]{trefethen2020spectra}. The operator-level definitions of pseudospectra are given in \cite[§4, p.~31, (4.3)-(4.5), Thm.~4.3]{trefethen2020spectra}.
In this work we use the closed CH23 convention
\[
\sigma_\varepsilon(A) = \overline{\{z\in\mathbb C : \|(A-zI)^{-1}\|>\varepsilon^{-1}\}},
\]
cf. \cite[§3, p.~9]{ColbrookHansen22}. The required operator identities are elementary consequences of the resolvent formulae: if \(B=UAU^{-1}\) with \(U\) unitary, then
\[
B-zI=U(A-zI)U^{-1}, \qquad (B-zI)^{-1}=U(A-zI)^{-1}U^{-1},
\]
and hence
\[
\sigma(B)=\sigma(A), \qquad \sigma_\varepsilon(B)=\sigma_\varepsilon(A).
\]
If \(B=A\oplus T_0\), then
\[
\rho(B)=\rho(A)\cap\rho(T_0), \qquad (B-zI)^{-1}=(A-zI)^{-1}\oplus (T_0-zI)^{-1},
\]
and therefore
\[
\|(B-zI)^{-1}\| = \max\{\|(A-zI)^{-1}\|,\|(T_0-zI)^{-1}\|\}.
\]
Consequently,
\[
\sigma(A\oplus T_0)=\sigma(A)\cup\sigma(T_0), \qquad \sigma_\varepsilon(A\oplus T_0) = \sigma_\varepsilon(A)\cup\sigma_\varepsilon(T_0).
\]

\begin{lemma}[Geometric encodings preserve both window predicates]\label{lem:geom-preserves}
	Let
	\[
	E:\mathcal X_\rho(J)\to\mathcal X_{\rho'}(J)
	\]
	be a CH23 \((J,\varepsilon)\)-neutral geometric encoding, and write
	\[
	E(A,K)=(A_E,K).
	\]
	Then for every admissible input
	\[
	(A,K)\in\mathcal X_\rho(J), \qquad K=\{z\}\subset J,
	\]
	one has
	\[
	\sigma(\operatorname{op}_{\rho'}(A_E))\cap K=\varnothing \quad\Longleftrightarrow\quad
	\sigma(\operatorname{op}_{\rho}(A))\cap K=\varnothing,
	\]
	and
	\[
	\sigma_\varepsilon(\operatorname{op}_{\rho'}(A_E))\cap K=\varnothing \quad\Longleftrightarrow\quad
	\sigma_\varepsilon(\operatorname{op}_{\rho}(A))\cap K=\varnothing.
	\]
\end{lemma}

\begin{proof}
	It suffices to check the claim for each elementary encoding and then use induction over finite compositions.
	
	For unitary conjugacy, basis relabeling, graph relabeling, and representation inclusion, there is a fixed unitary
	\[
	U:H_\rho\to H_{\rho'}
	\]
	such that, for the encoded instance \(E(A,K)=(A_E,K)\),
	\[
	\operatorname{op}_{\rho'}(A_E) = U\,\operatorname{op}_\rho(A)\,U^{-1}.
	\]
	Hence, for every \(z\in\mathbb C\),
	\[
	\operatorname{op}_{\rho'}(A_E)-zI = U(\operatorname{op}_\rho(A)-zI)U^{-1}.
	\]
	Therefore
	\[
	\sigma(\operatorname{op}_{\rho'}(A_E)) = \sigma(\operatorname{op}_{\rho}(A)),
	\]
	and, on the common resolvent set,
	\[
	(\operatorname{op}_{\rho'}(A_E)-zI)^{-1} =
	U(\operatorname{op}_{\rho}(A)-zI)^{-1}U^{-1}.
	\]
	Thus
	\[
	\|(\operatorname{op}_{\rho'}(A_E)-zI)^{-1}\| =
	\|(\operatorname{op}_{\rho}(A)-zI)^{-1}\|,
	\]
	and consequently
	\[
	\sigma_\varepsilon(\operatorname{op}_{\rho'}(A_E)) =
	\sigma_\varepsilon(\operatorname{op}_{\rho}(A)).
	\]
	Since the compact input \(K\) is unchanged, both window predicates are preserved.
	
	For neutral direct-sum stabilization, there are fixed data \(T_0,W\) such that
	\[
	\operatorname{op}_{\rho'}(A_E) = W(\operatorname{op}_\rho(A)\oplus T_0)W^{-1}.
	\]
	By unitary invariance and the direct-sum resolvent identity,
	\[
	\sigma(\operatorname{op}_{\rho'}(A_E)) = \sigma(\operatorname{op}_{\rho}(A))\cup\sigma(T_0),
	\]
	and
	\[
	\sigma_\varepsilon(\operatorname{op}_{\rho'}(A_E)) = \sigma_\varepsilon(\operatorname{op}_{\rho}(A))\cup\sigma_\varepsilon(T_0).
	\]
	The \((J,\varepsilon)\)-neutrality assumption gives
	\[
	\sigma(T_0)\cap K=\varnothing, \qquad
	\sigma_\varepsilon(T_0)\cap K=\varnothing
	\]
	for every singleton \(K\subset J\). Hence both window predicates are preserved.
	
	Finite compositions preserve the two predicates because each elementary encoding preserves them. This proves the lemma.
\end{proof}

Thus a geometric encoding cannot turn two instances with the same pseudospectral window answer into different target pseudospectral answers, nor can it do so for the exact
spectral window answer. This makes the cross-block separation proofs very short.

\subsection{Within-Block Reductions And Cross-Block Geometric Separations}
First we check the positive reductions inside each block. These are exactly the natural representation embeddings.

\begin{lemma}[Within-block geometric source reductions]\label{lem:within-block-geom}
	One has
	\[
	\Psd\lemfq{\geom_{J,\varepsilon}^{\mathrm{CH23}}}\Psd,
	\qquad
	\Psd\lemfq{\geom_{J,\varepsilon}^{\mathrm{CH23}}}\Psg,
	\qquad
	\Psd\lemfq{\geom_{J,\varepsilon}^{\mathrm{CH23}}}\Psr,
	\]
	and
	\[
	\Ppd\lemfq{\geom_{J,\varepsilon}^{\mathrm{CH23}}}\Ppd,
	\qquad
	\Ppd\lemfq{\geom_{J,\varepsilon}^{\mathrm{CH23}}}\Ppg,
	\qquad
	\Ppd\lemfq{\geom_{J,\varepsilon}^{\mathrm{CH23}}}\Ppr.
	\]
\end{lemma}

\begin{proof}
	The diagonal-to-diagonal reductions are reflexivity.
	
	For diagonal-to-general, use the representation-inclusion encoding
	\[
	E(A,K):=(A,K),
	\]
	where the same diagonal operator is viewed as a member of the general \(\ell^2(\mathbb N)\)-operator class. This encoding is one of the generators of $\geom^{\mathrm{CH23}}_{J,\varepsilon}$. The decoder is
	\[
	D=\operatorname{id}_{\{0,1\}}.
	\]
	By representation inclusion, the underlying operator and compact input are unchanged. Hence the exact spectral and fixed-\(\varepsilon\) pseudospectral target identities hold.
	
	For the finite-query reconstruction, a target matrix-entry query satisfies
	\[
	\langle Ae_j,e_i\rangle =
	\begin{cases}
		\mu_j(A,K),&i=j,\\
		0,&i\neq j.
	\end{cases}
	\]
	Compact-window approximants are passed through unchanged. The general-class auxiliary bounded-dispersion data can be chosen canonically for diagonal operators, for example
	\[
	f(n)=n, \qquad c_n=2^{-n},
	\]
	because
	\[
	(I-P_n)AP_n=0, \qquad (I-P_n)A^*P_n=0.
	\]
	Thus all target evaluations are reconstructed from finitely many source evaluations and constants. This proves both
	\[
	\mathcal P^\sigma_{J,\diag} \le_{\geom^{\mathrm{CH23}}_{J,\varepsilon},\mathrm{fq}} \mathcal P^\sigma_{J,\gen}
	\]
	and
	\[
	\mathcal P^{\sigma_\varepsilon}_{J,\varepsilon,\diag} \le_{\geom^{\mathrm{CH23}}_{J,\varepsilon},\mathrm{fq}}
	\mathcal P^{\sigma_\varepsilon}_{J,\varepsilon,\gen}.
	\]
	
	For diagonal-to-graph, fix a connected countable graph
	\[
	\mathcal G=(V,E_{\mathcal G}), \qquad V=\{v_1,v_2,\ldots\},
	\]
	and the unitary
	\[
	U:\ell^2(\mathbb N)\to\ell^2(V), \qquad Ue_j=\mathbf 1_{v_j}.
	\]
	Encode by
	\[
	E(A,K):=(UAU^{-1},K).
	\]
	This is a composition of a unitary relabeling and a graph-representation inclusion, hence is \(\geom^{\mathrm{CH23}}_{J,\varepsilon}\)-admissible.
	
	If
	\[
	Ae_j=a_je_j,
	\]
	then
	\[
	UAU^{-1}\mathbf 1_{v_j}=a_j\mathbf 1_{v_j}.
	\]
	Thus the graph coefficient function satisfies
	\[
	\alpha(v_i,v_j) =
	\begin{cases}
		a_j,&i=j,\\
		0,&i\neq j.
	\end{cases}
	\]
	Therefore each graph coefficient evaluation is reconstructed from either one diagonal source evaluation or a constant. Local support sets may be chosen as
	\[
	S_{v_j}:=\{v_j\},
	\]
	so local support queries are constant reconstructions. Compact-window approximants are unchanged by this procedure.
	
	The output identities follow from unitary invariance:
	\[
	\sigma(UAU^{-1})=\sigma(A), \qquad
	\sigma_\varepsilon(UAU^{-1}) = \sigma_\varepsilon(A).
	\]
	This proves the two diagonal-to-graph reductions.
\end{proof}

The cross-block directions fail because exact spectrum and fixed \(\varepsilon\)-pseudospectrum do not determine one another under predicate-preserving geometric encodings.

\begin{lemma}[Geometric cross-block separations]\label{lem:geom-cross-separations}
	For every \(\star\in\{\diag,\gen,\Graph\}\),
	\[
	\Psd \Nlemfq{\geom_{J,\varepsilon}^{\mathrm{CH23}}} \mathcal P^{\sigma_\varepsilon}_{J,\varepsilon,\star},
	\]
	and
	\[
	\Ppd \Nlemfq{\geom_{J,\varepsilon}^{\mathrm{CH23}}} \mathcal P^{\sigma}_{J,\star}.
	\]
\end{lemma}

\begin{proof}
	Choose
	\[
	x\in\operatorname{int}(J), \qquad K:=\{x\}.
	\]
	
	\smallskip
	\noindent \textbf{Step 1: No geometric reduction from the exact spectral source to a pseudospectral target:}
	Suppose, toward contradiction, that for some
	\[
	\star\in\{\diag,\gen,\Graph\}
	\]
	there is a geometric finite-query transport
	\[
	\mathcal P^\sigma_{J,\diag} \le_{\geom^{\mathrm{CH23}}_{J,\varepsilon},\mathrm{fq}} \mathcal P^{\sigma_\varepsilon}_{J,\varepsilon,\star}.
	\]
	Let $E$	be its geometric encoding and let $D:\{0,1\}\to\{0,1\}$	be its decoder.
	
	Consider the scalar diagonal operators
	\[
	A_0:=xI, \qquad A_1:=\left(x+\frac{\varepsilon}{2}\right)I.
	\]
	For the exact spectral source,
	\[
	\sigma(A_0)=\{x\}, \qquad \sigma(A_1)=\left\{x+\frac{\varepsilon}{2}\right\}.
	\]
	Thus
	\[
	\sigma(A_0)\cap K\neq\varnothing, \qquad \sigma(A_1)\cap K=\varnothing.
	\]
	With the $0/1$ convention this gives
	\[
	\Xi_\sigma(A_0,K)=0, \qquad \Xi_\sigma(A_1,K)=1.
	\]
	
	On the other hand, since \(A_0\) and \(A_1\) are normal scalar operators,
	\[
	\operatorname{dist}(x,\sigma(A_0))=0, \qquad \operatorname{dist}(x,\sigma(A_1))=\frac{\varepsilon}{2}.
	\]
	Hence
	\[
	x\in\sigma_\varepsilon(A_0), \qquad x\in\sigma_\varepsilon(A_1).
	\]
	By \cref{lem:geom-preserves}, the geometric encoding preserves the pseudospectral window predicate. Therefore
	\[
	\Xi_{\sigma_\varepsilon}(E(A_0,K))=0, \qquad \Xi_{\sigma_\varepsilon}(E(A_1,K))=0.
	\]
	The output identity of the transport would now require
	\[
	D(0)=0 \qquad\text{and}\qquad D(0)=1,
	\]
	which is impossible. Therefore
	\[
	\mathcal P^\sigma_{J,\diag} \not\le_{\geom^{\mathrm{CH23}}_{J,\varepsilon},\mathrm{fq}} \mathcal P^{\sigma_\varepsilon}_{J,\varepsilon,\star}.
	\]
	
	\smallskip
	\noindent \textbf{Step 2: No geometric reduction from the pseudospectral source to an exact spectral target:}
	Suppose, toward contradiction, that for some
	\[
	\star\in\{\diag,\gen,\Graph\}
	\]
	there is a geometric finite-query transport
	\[
	\mathcal P^{\sigma_\varepsilon}_{J,\varepsilon,\diag} \le_{\geom^{\mathrm{CH23}}_{J,\varepsilon},\mathrm{fq}} \mathcal P^\sigma_{J,\star}.
	\]
	Let again $E$ be its geometric encoding and $D:\{0,1\}\to\{0,1\}$ its decoder.
	
	Define
	\[
	B_0:=\left(x+\frac{\varepsilon}{2}\right)I, \qquad B_1:=(x+2\varepsilon)I.
	\]
	Then
	\[
	\sigma(B_0)\cap K=\varnothing, \qquad \sigma(B_1)\cap K=\varnothing.
	\]
	Hence the exact spectral target output is
	\[
	\Xi_\sigma(B_0,K)=1, \qquad \Xi_\sigma(B_1,K)=1.
	\]
	By \cref{lem:geom-preserves}, the geometric encoding preserves this exact spectral window predicate, so
	\[
	\Xi_\sigma(E(B_0,K))=1, \qquad \Xi_\sigma(E(B_1,K))=1.
	\]
	
	For the pseudospectral source,
	\[
	\operatorname{dist}(x,\sigma(B_0))=\frac{\varepsilon}{2} \le\varepsilon,
	\]
	so
	\[
	x\in\sigma_\varepsilon(B_0), \qquad \Xi_{\sigma_\varepsilon}(B_0,K)=0.
	\]
	But
	\[
	\operatorname{dist}(x,\sigma(B_1))=2\varepsilon>\varepsilon,
	\]
	so
	\[
	x\notin\sigma_\varepsilon(B_1), \qquad \Xi_{\sigma_\varepsilon}(B_1,K)=1.
	\]
	The output identity of the transport would now require
	\[
	D(1)=0 \qquad\text{and}\qquad D(1)=1,
	\]
	which is again impossible. Therefore
	\[
	\mathcal P^{\sigma_\varepsilon}_{J,\varepsilon,\diag} \not\le_{\geom^{\mathrm{CH23}}_{J,\varepsilon},\mathrm{fq}}
	\mathcal P^\sigma_{J,\star}.
	\]
\end{proof}

The proof book for the finite theorem-block theorem is therefore now complete.

\subsection{The Geometric Two-Source Theorem}
We assemble the within-block reductions and cross-block separations using the abstract finite source-separated theorem-block criterion.

\begin{theorem}[Geometric two-source theorem for the CH23 singleton block]\label{thm:geom-two-source}
	The six-problem ambient \(\mathcal U^{\sharp}_{2,J,\varepsilon}\) carries a \(\geom_{J,\varepsilon}^{\mathrm{CH23}}\)-source-separated finite theorem-block structure at height \(2\) with blocks
	\[
	\mathcal U^{\sigma,\mathrm{sgl}}_J \qquad\text{and}\qquad \mathcal U^{\sigma_\varepsilon,\mathrm{sgl}}_{J,\varepsilon},
	\]
	and sources
	\[
	\Psd, \qquad \Ppd.
	\]
	Consequently,
	\[
	\operatorname{MinDeg}^{\geom_{J,\varepsilon}^{\mathrm{CH23}}}_2 (\mathcal U^{\sharp}_{2,J,\varepsilon}) = 
	\left\{ [\mathcal P^\sigma_{J,\diag}]_{\geom^{\mathrm{CH23}}_{J,\varepsilon}}, [\mathcal P^{\sigma_\varepsilon}_{J,\varepsilon,\diag}]_{\geom^{\mathrm{CH23}}_{J,\varepsilon}} \right\}.
	\]
\end{theorem}

\begin{proof}
	This follows directly from \cref{thm:height-input}, \cref{lem:within-block-geom} and \cref{lem:geom-cross-separations}. Hence the hypotheses of \cref{thm:finite-block} hold with \(r=2\).
\end{proof}

\begin{corollary}[Raw-principal but geometrically two-source]\label{cor:raw-principal-geom-two-source}
	The CH23 singleton-window spectral/pseudospectral ambient satisfies
	\[
	\left| \operatorname{MinDeg}^{\raw}_2(\mathcal U^{\sharp}_{2,J,\varepsilon}) \right|=1,
	\]
	but
	\[
	\left| \operatorname{MinDeg}^{\geom_{J,\varepsilon}^{\mathrm{CH23}}}_2 (\mathcal U^{\sharp}_{2,J,\varepsilon}) \right|=2.
	\]
	Thus it is raw-principal but geometrically non-principal.
\end{corollary}

\begin{proof}
	This is a combination of \cref{cor:raw-principal} and \cref{thm:geom-two-source}.
\end{proof}

This is the main calibration phenomenon: the same exact SCI ambient has one minimal source in the raw degree theory and two minimal sources in the geometric degree theory.
The corrected exact-basis problem must therefore be modality-indexed.

\section{Modal Profiles And The Corrected Exact-Basis Problem}\label{sec:corrected-open-problem}
The CH23 theorem block suggests that the invariant of interest is not one degree set but a profile of degree sets across modalities.

\subsection{Refinement Maps And Modal Degree Profiles}
Whenever one modality is stricter than another, there is a canonical quotient map from the stricter degree set to the coarser one.

\begin{definition}[Modal refinement map]\label{def:refinement-map}
	Let \(\mathcal U\) be a family of (typed) SCI computational problems, let \(k\in\bN\), and let \(\mathfrak m\preceq\mathfrak n\). Then there is a canonical map
	\[
	\pi_{\mathfrak m\to\mathfrak n}: D^{\mathfrak m}_k(\mathcal U) \to D^{\mathfrak n}_k(U), \qquad
	[\mathcal P]_{\mathfrak m} \mapsto [\mathcal P]_{\mathfrak n}.
	\]
	It is well defined because \(\mathcal P \equiv_{\mathfrak m,\fq} \mathcal Q\) implies \(\mathcal P \equiv_{\mathfrak n,\fq} \mathcal Q\). Its fibers measure which distinctions are visible in the stricter modality \(\mathfrak m\) but invisible in the coarser modality \(\mathfrak n\).
\end{definition}

The fibers of this map record precisely which distinctions are forgotten when one moves to a coarser semantic world.

\begin{definition}[Raw-collapsed but modally split ambient]\label{def:raw-collapsed-modally-split}
	Let \(\mathcal U\) be \(k\)-bounded in the sense that \(\SCIG(\mathcal P)\le k\) for all \(\mathcal P\in \mathcal U\), and let \(\mathfrak m\preceq\raw\). We say that \(\mathcal U\) is raw-collapsed but \(\mathfrak m\)-split at height \(k\) if there exist
	\[
	\mathcal P, \mathcal Q\in\cO_k(\mathcal U)
	\]
	such that
	\[
	\mathcal P\equiv_{\raw,\fq} \mathcal Q \qquad\text{but}\qquad \mathcal P\not\equiv_{\mathfrak m,\fq} \mathcal Q.
	\]
	Equivalently, some fiber of
	\[
	\pi_{\mathfrak m\to\raw}: D^{\mathfrak m}_k(\mathcal U) \to D^{\raw}_k(\mathcal U)
	\]
	has cardinality at least two.
\end{definition}

The CH23 diagonal spectral/pseudospectral sources manifest a prototype for this point: they are equivalent raw-finitely, but not geometrically equivalent.

\begin{definition}[Basis-refining ambient]\label{def:basis-refining}
	Let \(\mathfrak m\preceq\mathfrak n\). A \(k\)-bounded ambient \(\mathcal U\) is basis-refining from \(\mathfrak n\) to \(\mathfrak m\) at height \(k\) if
	\[
	\left| \operatorname{MinDeg}^{\mathfrak m}_k(\mathcal U) \right| > \left| \operatorname{MinDeg}^{\mathfrak n}_k(\mathcal U) \right|.
	\]
	The CH23 ambient \(\mathcal U^{\sharp}_{2,J,\varepsilon}\) is basis-refining from \(\raw\) to \(\geom_{J,\varepsilon}^{\mathrm{CH23}}\) by \cref{cor:raw-principal-geom-two-source}.
\end{definition}

Basis-refinement is stronger than mere splitting: it says that the number of minimal sources needed to cover the exact layer changes with the modality.

\subsection{The Corrected Modal Exact-Basis Problem}
We can now state the corrected form of the exact-basis problem. The word ``natural'' is external: it refers to mathematical origin, not to whether a family happens to split
modally.

\begin{openproblem}[Modal exact-basis and finite-rooted ambient problem]\label{op:modal-open-problem-1}
	Let \(\mathcal N\) be an externally specified class of natural SCI ambients. ``Natural'' is not defined by raw collapse or modal separation; it is supplied by the mathematical source of the problems, such as spectral decision problems, pseudospectral decision problems, spectral-gap and bottom-classification problems, spectral-measure/type problems, information-regime variants, Koopman spectral tasks, and other SCI families arising in the literature.
	
	For each raw-sound admissible finite-query modality
	\[
	\mathfrak m\preceq\raw,
	\]
	each \(\mathcal U\in\mathcal N\), and each \(k\in\mathbb N\), determine the modal exact degree structure
	\[
	D^{\mathfrak m}_k(\mathcal U) := \cO_k(\mathcal U)/{\equiv_{\mathfrak m,\fq}},
	\]
	construct a modal exact basis
	\[
	B^{\mathfrak m}_k(\mathcal U)\subseteq\cO_k(\mathcal U),
	\]
	and classify it as principal, finitely generated non-principal, or non-finitely generated. If one imposes an additional external naturality condition on admissible sources, then a natural basis may fail to exist inside \(\mathcal U\); in that strengthened sense one may ask whether the ambient must be enlarged.
	
	For modalities not below \(\raw\), the corresponding exact-basis problem should be formulated relative to an implementation height \(\mathrm{SCI}_C\) for which \(\mathfrak m\) is sound.
	
	Moreover, for \(\mathfrak m\preceq\mathfrak n\), study the refinement maps
	\[
	\pi_{\mathfrak m\to\mathfrak n}: D^{\mathfrak m}_k(\mathcal U) \to D^{\mathfrak n}_k(\mathcal U)
	\]
	and classify which natural ambients are raw-principal but modally non-principal, raw-collapsed but modally split, principal in all modalities, or non-principal in all modalities.
	
	Finally, construct finite global rooted ambients
	\[
	\mathcal F \longmapsto \mathcal U^{\sharp,\mathfrak m}_k(\mathcal F)
	\]
	for natural families \(\mathcal F\), with modal minimal-support maps, such that modal witness-space sharpness gives a necessary-and-sufficient criterion for family-pointwise exactness at height \(k\).
\end{openproblem}

\begin{remark}[What is solved here]\label{rem:what-solved}
	This paper solves the first nontrivial finite instance of \cref{op:modal-open-problem-1}: for the CH23 singleton-window spectral/pseudospectral ambient, the raw exact basis is principal, while the CH23 geometric exact basis has two minimal sources. It does not solve/prove the universal covering theorem for all natural SCI families. It supplies the corrected formulation and the first calibrated theorem-block example.
\end{remark}

\begin{remark}[Relation to implementation modalities]\label{rem:implementation-modalities}
	Implementation modalities, such as Type-2, Weihrauch, BSS, Borel, or continuous tower implementations, determine which approximation towers count as implemented solutions. Transport modalities determine which reductions between problems count. A pair consisting of an implementation modality \(C\) and a transport modality \(\mathfrak m\) is sound when
	\[
	\mathcal S \lemfq{\mathfrak m} \mathcal P \quad\Longrightarrow\quad \SCI_C(\mathcal S) \le \SCI_C(\mathcal P).
	\]
	The raw finite-query preorder is sound for raw type-G SCI. For implemented SCI notions, one must impose corresponding regularity and uniformity on encodings, decoders, and reconstruction maps. This is why the modality-indexed formulation is not redundant with the implementation hierarchy.
\end{remark}

\subsection{Semantic Interpretation: Raw Truth, Computational Meaning, And Geometry}
There are three distinct senses of ``meaning'' in the modal SCI picture.

\smallskip

\noindent\textbf{Formal-extensional meaning:}
A raw SCI computational problem
\[
\mathcal P=(\Xi,\Omega,\mathcal Y,\Lambda)
\]
already has formal meaning in the ordinary set-theoretic metatheory: it is a target map
\[
\Xi: \Omega\to \mathcal Y
\]
together with an evaluation interface. A statement such as
\[
\mathcal S \lemfq{\raw} \mathcal P
\]
has a definite Tarskian truth condition in the metatheory in the sense of the notation of truth in \cite{Tarski1944}. In this sense, the raw level is not meaningless. It is a logical framework in which the sentence
\[
``\mathcal S \lemfq{\raw} \mathcal P"
\]
is true exactly when the corresponding encoding, decoder, and finite reconstruction data exist.

This is the sense in which one may say that raw-equivalent problems have the same \textit{raw finite-information truth value}: each can be simulated from the other inside
the raw framework.

\smallskip

\noindent\textbf{Modal or implemented meaning:}
A modality $\mathfrak m$ adds semantic structure to the raw syntax by specifying which encodings, decoders, and finite transcript reconstructions count as admissible. For example, the TTE modality interprets the raw problem inside represented spaces and requires computable realizers; the geometric modality interprets the same raw problem inside a category of natural operator-theoretic transformations.

Thus it does not follow here that TTE is a semantics of all set theory. It is a represented-space semantics for the raw SCI problem language. It supplies an implementation world in which a raw map is accepted only if it has a uniform computable realizer.

In this sense, modalities are semantic enrichments of the raw SCI syntax. They do not replace the raw problem; they specify which interpretations of its maps are admissible.

\smallskip

\noindent \textbf{Modal degree profile as computational meaning:}
The CH23 spectral/pseudospectral example shows that two problems may be equivalent in the raw modality but distinct in a stricter modality:
\[
\mathcal P^\sigma_{J,\diag} \equiv_{\raw,\fq}
\mathcal P^{\sigma_\varepsilon}_{J,\varepsilon,\diag},
\]
while, under the CH23 geometric modality,
\[
\mathcal P^\sigma_{J,\diag} \not\equiv_{\geom,\fq}
\mathcal P^{\sigma_\varepsilon}_{J,\varepsilon,\diag}.
\]
Therefore the raw degree does not exhaust computational meaning. The more stable invariant is the modal degree profile
\[
\mathfrak m \longmapsto [\mathcal P]_{\mathfrak m}.
\]
This profile records which equivalences survive when one changes the semantic world.

\smallskip

This interpretation is close to Carnap's idea that questions become internal and truth-evaluable only after choosing a linguistic framework, while the choice of framework
itself is a practical or methodological decision, see \cite{Carnap1950}. It is also compatible with Quine's warning that meaning should not be reduced to isolated analytic equivalence or statement-by-statement reduction, see \cite{quine2000two}. In the present setting, the raw framework gives one internal notion of equivalence, but stricter modalities reveal distinctions hidden by that raw equivalence.

So the modal SCI conclusion is not the ``raw problems have no meaning.'' The correct conclusion is more that ``raw equivalence is only one framework-internal meaning; computational meaning is modality-relative.''

\section{Outlook}
The corrected modal exact-basis problem suggests three immediate directions. First, one should analyze information-regime ambients, where raw non-reductions are expected to be robust under all stricter raw-sound modalities.  Second, one should formulate Koopman geometric modalities (as these represent an important computational family in the SCI literature) and determine which restricted Koopman families admit finite modal exact bases. Promising candidate problem families are the Koopman $L^2$ case from \cite{colbrook2024limits}, the $L^p$, $1<p<\infty$, case from \cite{sorg2025solvability} and the RKHS case from \cite{boulle2025convergent}. Third, beyond the trace characterization and strictness result proved in \cref{subsec:TTEfqtransportsW}, one should compare the TTE finite-query exact-degree structure with Weihrauch-theoretic ranks, cylinders, jumps, and witness-uniformity phenomena. In particular, it remains open whether natural SCI exact sources have recognizable Weihrauch normal forms after forgetting interface data, and which modal distinctions disappear under this forgetful passage. These directions are left for future work.

\textbf{Acknowledgments } The author thanks Patrick Uftring for interesting discussions on the correct formulation of \cref{def:regularity-controlled-modalities}(v).

\textbf{Funding } This research did not receive any specific grant from funding agencies in the public, commercial, or not-for-profit sectors.

\textbf{Statement } During the preparation of this work the author used UniBwM-ChatGPT5.5 in order to improve the language in the abstract and introduction. After using this tool, the author reviewed and edited the content as needed and takes full responsibility for the content of the published article.

\textbf{Conflict of interest } The author declares no conflict of interest.

\bibliographystyle{alpha}  
\bibliography{bib.bib}

\newcommand{\etalchar}[1]{$^{#1}$}
\begin{thebibliography}{BACH{\etalchar{+}}15}

\bibitem[BACH{\etalchar{+}}15]{ben2015computing}
Jonathan Ben-Artzi, Matthew~J Colbrook, Anders~C Hansen, Olavi Nevanlinna, and
  Markus Seidel.
\newblock Computing spectra--on the solvability complexity index hierarchy and
  towers of algorithms.
\newblock {\em arXiv preprint arXiv:1508.03280}, 2015.

\bibitem[BCC25]{boulle2025convergent}
Nicolas Boull{\'e}, Matthew~J Colbrook, and Gustav Conradie.
\newblock Convergent methods for koopman operators on reproducing kernel
  hilbert spaces.
\newblock {\em arXiv preprint arXiv:2506.15782}, 2025.

\bibitem[BGP21]{BrattkaGherardiPauly2021}
Vasco Brattka, Guido Gherardi, and Arno Pauly.
\newblock Weihrauch complexity in computable analysis.
\newblock In {\em Handbook of computability and complexity in analysis}, pages
  367--417. Springer, 2021.

\bibitem[BP18]{BrattkaPauly2018}
Vasco Brattka and Arno Pauly.
\newblock On the algebraic structure of {W}eihrauch degrees.
\newblock {\em Log. Methods Comput. Sci.}, 14(4):Paper No. 4, 36, 2018.

\bibitem[Car50]{Carnap1950}
Rudolf Carnap.
\newblock Empiricism, semantics, and ontology.
\newblock {\em Revue internationale de philosophie}, pages 20--40, 1950.

\bibitem[CH23]{ColbrookHansen22}
Matthew~J. Colbrook and Anders~C. Hansen.
\newblock The foundations of spectral computations via the solvability
  complexity index hierarchy.
\newblock {\em J. Eur. Math. Soc. (JEMS)}, 25(12):4639--4718, 2023.

\bibitem[CMS24]{colbrook2024limits}
Matthew~J Colbrook, Igor Mezi{\'c}, and Alexei Stepanenko.
\newblock Limits and powers of koopman learning.
\newblock {\em arXiv preprint arXiv:2407.06312}, 2024.

\bibitem[Han11]{hansen2011solvability}
Anders Hansen.
\newblock On the solvability complexity index, the $\varepsilon$-pseudospectrum
  and approximations of spectra of operators.
\newblock {\em Journal of the American Mathematical Society}, 24(1):81--124,
  2011.

\bibitem[Kec95]{Kechris1995}
Alexander~S. Kechris.
\newblock {\em Classical descriptive set theory}, volume 156 of {\em Graduate
  Texts in Mathematics}.
\newblock Springer-Verlag, New York, 1995.

\bibitem[Qui00]{quine2000two}
Willard Van~Orman Quine.
\newblock Two dogmas of empiricism.
\newblock {\em Perspectives in the Philosophy of Language}, pages 189--210,
  2000.

\bibitem[Sor25]{sorg2025solvability}
Christopher Sorg.
\newblock Residual sci upper bounds and lower witnesses for koopman approximate
  point spectra in ${L}^p$ for $1<p<\infty$: Extended version.
\newblock {\em arXiv preprint arXiv:2509.16016v3}, 2025.

\bibitem[Sor26a]{Sorg26Foundations}
Christopher Sorg.
\newblock Foundational analysis of the solvability complexity index: The
  weihrauch-sci intermediate hierarchy.
\newblock {\em arXiv preprint arXiv:2603.18955v2}, 2026.

\bibitem[Sor26b]{Sorg26Witness}
Christopher Sorg.
\newblock From witness-space sharpness to family-pointwise exactness for the
  solvability complexity index.
\newblock {\em arXiv preprint arXiv:2604.12750v2}, 2026.

\bibitem[Tar44]{Tarski1944}
Alfred Tarski.
\newblock The semantic conception of truth and the foundations of semantics.
\newblock {\em Philos. and Phenomenol. Res.}, 4:341--376, 1944.

\bibitem[TE05]{trefethen2020spectra}
Lloyd~N. Trefethen and Mark Embree.
\newblock {\em Spectra and pseudospectra}.
\newblock Princeton University Press, Princeton, NJ, 2005.
\newblock The behavior of nonnormal matrices and operators.

\bibitem[Wei00]{weihrauch2000computable}
Klaus Weihrauch.
\newblock {\em Computable analysis: an introduction}.
\newblock Springer Science \& Business Media, 2000.

\end{thebibliography}

\end{document}